\documentclass[a4paper,oneside,reqno]{amsart}

\usepackage{amssymb}

\usepackage{enumitem}

\usepackage{mathptmx}

\usepackage{hyperref}
\hypersetup{colorlinks=true,linkcolor=RoyalPurple,citecolor=RoyalPurple}

\usepackage{cite}

\usepackage{geometry}

\usepackage{color}
\usepackage[dvipsnames]{xcolor}

\let\oldsqrt\sqrt
\def\sqrt{\mathpalette\DHLhksqrt}
\def\DHLhksqrt#1#2{%
\setbox0=\hbox{$#1\oldsqrt{#2\,}$}\dimen0=\ht0
\advance\dimen0-0.2\ht0
\setbox2=\hbox{\vrule height\ht0 depth -\dimen0}%
{\box0\lower0.4pt\box2}}

\allowdisplaybreaks

\newcommand{\be}{\begin{equation}}
\newcommand{\ee}{\end{equation}}

\newcommand{\R}{\mathbb{R}} 
\newcommand{\N}{\mathbb{N}} 

\newcommand{\supp}{\textnormal{supp}} 

\renewcommand{\div}{\textnormal{div}}

\newcommand{\n}{\nabla}
\newcommand{\D}{\Delta}
\newcommand{\Ds}{(-\Delta)^s}
\newcommand{\Rp}{\mathbb{R}^N_+}

\newcommand{\cF}{{\mathcal F}}
\newcommand{\cG}{{G}}

\newcommand{\cL}{{\mathcal L}}

\newcommand{\eps}{\varepsilon}

\theoremstyle{definition}
\newtheorem{defi}{Definition}[section]
\newtheorem{remark}[defi]{Remark}

\theoremstyle{plain} 
\newtheorem{thm}[defi]{Theorem}
\newtheorem{prop}[defi]{Proposition}
\newtheorem{lemma}[defi]{Lemma}
\newtheorem{cor}[defi]{Corollary}

\theoremstyle{definition}

\numberwithin{equation}{section}

\date{\today}
\title{Positive powers of the Laplacian in the half-space under Dirichlet boundary conditions}
\author{N. Abatangelo}
\address{N. Abatangelo {\rm D\'{e}partement de math\'{e}matique, Universit\'{e} Libre de Bruxelles CP 214, Boulevard du Triomphe, 1050 Ixelles, Belgium.}}
\email{nicola.abatangelo@ulb.ac.be}
\author{S. Dipierro}
\address{S. Dipierro {\rm Dipartimento di Matematica, Universit\`a degli Studi di Milano, via Saldini 50, 20133 Milano, Italy, and School of Mathematics and Statistics, University of Western Australia, 35 Stirling Hwy, Crawley WA 6009, Australia.}}
\email{serena.dipierro@unimi.it}
\author{M.M. Fall}
\address{M.M. Fall {\rm African Institute for Mathematical Sciences (A.I.M.S) of Senegal, KM 2, Route de Joal (Centre I.R.D. Mbour), B.P. 1418 Mbour, S\'en\'egal.}}
\email{mouhamed.m.fall@aims-senegal.org}
\author{S. Jarohs}
\address{S. Jarohs {\rm Institut f\"ur Mathematik, Goethe-Universit\"at, Robert-Mayer-Stra\ss e 10, 60054 Frankfurt, Germany.}}
\email{jarohs@math.uni-frankfurt.de}
\author{A. Salda\~na}
\address{A. Salda\~na {\rm Institut f\"ur Analysis, Karlsruher Institut f\"ur Technologie, Englerstra\ss e 2, 76131 Karlsruhe, Germany.}}
\email{alberto.saldana@partner.kit.edu}

\begin{document}

\let\thefootnote\relax
\let\thefootnote\relax

\begin{abstract}
	We present explicit formulas for solutions to nonhomogeneous boundary value problems involving any positive power of the Laplacian in the half-space. For non-integer powers the operator becomes nonlocal and this requires a suitable extension of Dirichlet-type boundary conditions.  A key ingredient in our proofs is a point inversion transformation which preserves harmonicity and allows us to use known results for the ball.  We include uniqueness statements, regularity estimates, and describe the growth or decay of solutions at infinity and at the boundary.
\end{abstract}
\maketitle

\section{Introduction}

In this paper we study explicit formulas for solutions to nonhomogeneous boundary value problems for any positive power $s>0$ of the Laplacian $(-\Delta)^s$ in the half-space $\R^N_+:=\{x\in\R^N\::\: x_1>0\}$.
Such explicit formulas play a prominent role in Liouville-type theorems, scaling arguments, and in the study of qualitative properties (monotonicity, symmetry, \emph{etc.}) of solutions to nonlinear equations, see \cite{RW09,FW16, GGS10}.

In our results, we consider a consistent extension of Dirichlet boundary conditions to the higher-order fractional setting given by
\begin{align}\label{trace}
	D^{k+\sigma-1}u(z):=\frac{1}{k!}{\lim_{x_1\to 0^+}}\partial_1^{k}[x_1^{1-\sigma}u(x)]\qquad \text{ for }z\in \partial \Rp,
\end{align}
where $\sigma\in(0,1]$, $k\in\N$, and $x=(x_1,z)\in \Rp$.  Note that for $\sigma=1$, the traces \eqref{trace} reduce to the usual (inward) normal derivatives associated to  Dirichlet boundary conditions for the polyharmonic operator.

For $N\in\N$, $\sigma\in(0,1],$ $m\in\N_0$, $s=m+\sigma$, consider the problem 
\begin{equation}\label{eq:intro:main}
	\left\{\begin{aligned}
		(-\Delta)^su&=f &&\text{ in }\Rp,\\
		u&=g&& \text{ on }\R^N\backslash \Rp,\\
		D^{k+\sigma-1}u &= h_k&&\text{ on }\partial \Rp\ \text{ for $k\in\{0,\ldots,m\}$.}\end{aligned}\right. 
\end{equation}
If $\sigma=1$ then $(-\Delta)^s=(-\Delta)^{m+1}$ is the usual iterated Laplacian, that is,
\begin{align*}
	(-\Delta)^{m+1}u(x)=\Big(-\sum_{i=1}^{N}\partial_{i}^2\Big)^{m+1}u(x)\qquad \text{ for }x\in \R^N,
\end{align*}
and, if $\sigma\in(0,1)$, $(-\Delta)^s=(-\Delta)^{m+\sigma}$ becomes a nonlocal operator and can be evaluated pointwisely using finite differences
\begin{align}\label{Ds:def}
	(-\Delta)^s u(x):=\frac{c_{N,s}}{2}\int_{\R^N} \frac{\sum_{k=-{m-1}}^{m+1} (-1)^k { \binom{2m+2}{m+1-k}} u(x+ky)}{|y|^{N+2s}} \ dy\qquad \text{ for }x\in \R^N,
\end{align}
where $c_{N,s}$ is a positive normalization constant whose precise value can be found in \eqref{cNms:def} below. The integral \eqref{Ds:def} is finite whenever $u$ is locally $C^{2s+\alpha}$ and belongs to $\cL^1_s$ (see the notation section below for definitions), but it cannot be computed explicitly in general and this is one of the main difficulties in the nonlocal setting.  For more information on the pointwise evaluation \eqref{Ds:def}, see \cite{AJS17b}.

Since $(-\Delta)^s$ is a nonlocal operator for noninteger $s$, in this case one can prescribe data also on the complement of the half-space. 
In the following we use $(x_1)^{\alpha}_+$ to denote the function $0$ if $x_1\leq 0$ (also if $\alpha<0$) and $x_1^{\alpha}$ if $x_1>0$.  Our first result shows that the \textit{nonlocal Poisson kernel for the half-space} is given by
\begin{equation}\label{s-poisson}
	\Gamma_s(x,y):=(-1)^m\gamma_{N,\sigma}\frac{(x_1)_+^s}{(-y_1)^s|x-y|^N}\qquad \text{for $x\in\R^N$ and $y\in \R^N\backslash \overline{\Rp}$},
\end{equation}
where  $\gamma_{N,\sigma}$ is a positive normalization constants, see \eqref{constants} below.  Observe that the kernel $\Gamma_s$ alternates sign depending on the parity of $m$.  The case $m=0$, \emph{i.e.,} $s\in(0,1)$, is remarked in \cite[equation (3.40)]{B99}. 

\begin{thm}\label{thm:Poisson}
	Let $s\in(0,\infty)\backslash \N$, $r>0$, $g\in L^1(\R^N)$ such that $g$ has compact support in $\R^N\backslash \overline{\Rp}$ and let $u:\R^N\to \R$ be given by
	\begin{align}\label{cHdef}
		u(x)
		:=\int_{\R^N\setminus\Rp}\Gamma_s(x,y) g(y)\ dy+g(x)\chi_{\R^N\setminus\overline{\Rp}}(x)\qquad \text{ for }x\in\R^N.
	\end{align}
	Then $u\in C^{\infty}(\Rp)\cap C^{s}_0(\Rp)\cap \cL^1_s$ is a solution of 
	\begin{align}\label{eq0}
		(-\Delta)^s u=0\quad \text{ in }\Rp,\qquad u=g\quad \text{ on }\R^N\backslash\Rp,
	\end{align}
	and there is some $C(N,g,s)=C>0$ such that
	\begin{align}\label{bound}
		|u(x)|\leq C\frac{x_1^s}{1+|x|^N} \qquad \text{ for }x\in\Rp.
	\end{align}
	Moreover, $u$ is the unique solution of \eqref{eq0} satisfying \eqref{bound}.  If, in addition, $g\neq 0$ is nonnegative then there is $C>0$ depending on $g$, $N$, and $s$ such that
	\begin{align}\label{cor:bound}
		C^{-1}\frac{x_1^s}{1+|x|^N}\leq (-1)^mu(x)\leq C\frac{x_1^s}{1+|x|^N} \qquad \text{ for }x\in\Rp.
	\end{align}
\end{thm}
In particular, $u$ given by \eqref{cHdef} is a solution of  \eqref{eq:intro:main} with $f=h_k=0$ for all $k\in\{0,\ldots,m\}$. Note also that, by \eqref{cor:bound}, $u$ can be negative even if the data of the problem is nonnegative.  The assumption that $g=0$ in $\{x\in \R^N\;:\; x_1>-r\}$ is needed to guarantee integrability in \eqref{cHdef}; this assumption can be weakened, but data which is nonzero at $\partial \Rp$ necessarily requires a different kernel, see for example \cite[Theorem 1.6]{AJS17a} for the ball case.

Our next result concerns the Green function for $(-\Delta)^s$ in $\Rp$, which is defined by 
\begin{align}\label{greenhs0}
	\cG_s(x,y):=k_{N,s} |x-y|^{2s-N}\int_0^{\psi(x,y)}\frac{v^{s-1}}{(v+1)^\frac{N}{2}}\ dv,\qquad \psi(x,y):=\frac{4x_1^+y_1^+}{|x-y|^2}
\end{align}
for $x,y\in\R^N$, $x\neq y,$ where $k_{N,s}$ is a positive normalization constants, see \eqref{constants} below.

\begin{thm}\label{dist:sol:l}
	Let $f\in C^\beta_c(\Rp)$, $\beta>0$ such that $2s+\beta\notin\N$, and $u:\R^N\to \R$ be given by
	\begin{align*}
		u(x):=\int_{\Rp} \cG_s(x,y)f(y)\ dy\qquad \text{ for }x\in\R^N.
	\end{align*}
	Then $u\in C^{2s+\beta}(\Rp)$ is a solution of
	\begin{align}\label{eq}
		(-\Delta)^s u=f\quad \text{ in }\Rp,\qquad u=0\quad \text{ on }\R^N\backslash \Rp,
	\end{align}
	and there is some $C(N,f,s)=C>0$ such that
	\begin{align}\label{es}
		|u(x)|\leq C \frac{(x_1)_+^s}{1+|x|^{N}}\qquad \text{for all }x\in\R^N.
	\end{align}
	Moreover, $u$ is the unique solution of \eqref{eq} satisfying \eqref{es}.  
\end{thm}

That \eqref{greenhs0} is the Green function for $(-\Delta)^s$ in $\Rp$ was known for $s\in(0,1)\cup\N$, see \cite[equation (3.1)]{FW16} and \cite[Remark 2.28]{GGS10}.  Note that $G_s$ is a positive kernel, and therefore the half-space enjoys a positivity preserving property. For more information on maximum principles for $(-\Delta)^s$, see \cite{AJS17a,AJS16a}. We also point out that the Green function is \emph{not} uniquely determined, since one can always add suitable harmonic functions (see Proposition \ref{s-harmonic2} below); however, \eqref{greenhs0} is often referred to as \emph{the} Green function for the half-space (because of its relationship with the ball's Green function via the Kelvin transform), and in the following we use this convention as well.

Finally, we introduce the \textit{boundary Poisson kernels} for the half-space given by
\begin{align}\label{edenhofer}
	E_{k,s}(x,y)&:=\sum_{i = 0}^{\lfloor\frac{m-k}{2}\rfloor}\alpha_{m-k,i}\frac{(x_1)_+^{s+m-k-2i}}{|y-x|^{N+2(m-k-i)}},\qquad \text{for $k\in\{0,\ldots,m\}$, $y\in \partial\Rp$, $x\in \R^N\backslash\{y\}$,}
\end{align}
where $\alpha_{l,k}$ are normalization constants, see \eqref{constants} below, and $\lfloor a \rfloor$ denotes the integer part of $a$. For $s\in\N$, the kernels \eqref{edenhofer} were introduced in \cite{E77} using the equivalent expression
\begin{align}\label{Eden:form}
	E_{k,s}(x,y)=\alpha_{0,0}\, x_1^s\, D^{m-k}\zeta_x(y),\qquad \zeta_x(z):=|z-x|^{-N},\quad z\in\R^N\backslash\{x\}.
\end{align}
In \cite{E77} a representation formula for a restricted set of functions is shown using \eqref{Eden:form} and it is also stated the conjecture that these kernels provide pointwise solutions for nonhomogeneous Dirichlet polyharmonic problems (under some smoothness and growth assumptions on the data, see \cite[Satz 3]{E77}).  Our next theorem is new even in the local case and shows, in particular, that the conjecture in \cite[Satz 3]{E77} is true (at least) for compactly supported data. In the following $x'=0$ if $N=1$.

\begin{thm}
	\label{exp:sol:thm} 
	Let $m\in\N_0$, $\beta,\sigma\in(0,1]$, $s=m+\sigma$, $\eps>0$, $k\in \{0,\ldots,m\}$, $h_k\in C_c^{m+\beta}(\partial\Rp)$, and $u:\R^N\to\R$ be given by
	\begin{align}\label{u:eden}
		u(x)=\sum_{k=0}^m\int_{\partial\Rp}E_{k,s}(x,y) h_k(y)\ dy \qquad \text{ for }x\in\R^N.
	\end{align}
	Then $u\in C^{\infty}(\Rp)$ is a solution of
	\begin{align}\label{eq1}
		(-\Delta)^s u=0\quad \text{ in }\Rp,\qquad u=0\quad \text{ on }\R^N\backslash \Rp, \qquad D^{k+\sigma-1}u = h_k\text{ on }\partial \Rp
	\end{align}
	for $k\in\{0,\ldots,m\}$ and there is $C(N,h_1,\ldots,h_m,s)=C>0$ such that
	\begin{align}\label{es1}
		\begin{aligned}
			|\partial_1^m(x_1^{1-\sigma} u(x_1,x'))-D^{s-1}u(0,x')| &\leq C\frac{x_1^\beta}{1+|x'|^N}\quad \text{ for $x'\in\R^{N-1}$ and $x_1\in(0,1)$},\\
			|u(x)|&\leq C\frac{x_1^{s}}{1+|x|^{N}} \qquad \text{ for $x\in \R^N$ with $x_1\geq 1$,}
		\end{aligned}
	\end{align}
	Moreover, $u$ is the unique function in $C^{2s+\beta}(\Rp)$ satisfying \eqref{eq1} and \eqref{es1}.
\end{thm}
See also Theorem \ref{thm:edenhofer} below for more estimates on each of the kernels $E_{k,s}$.
We remark that the kernels $E_{m-1,s}$ and $E_{m,s}$ are connected via the trace operators with the Green function $\cG_s$,
see \eqref{first} and \eqref{second} below; however, the relationship between $E_{k,s}$ and $\cG_s$ is not so simple for $k\leq m-2$, see Remark \ref{ns:rem}. These identities are relevant to treat more general domains and for integration by parts formulas, for which only partial results are currently known in the fractional setting, see \cite{G2018, RS15}.  We also note that a boundary kernel for the half-space in the case $s\in(0,1)$ can be found in \cite[equation (3.38)]{B99}.

In general, one can find $s$-harmonic functions with much larger growth at infinity than those constructed in Theorem \ref{exp:sol:thm}, as the next proposition shows.
\begin{prop}\label{s-harmonic2}
	Let $m\in \N_0$, $\sigma\in(0,1]$, and $s=m+\sigma>0$. Then
	\begin{align*}
		\text{$x\mapsto (x_1)_+^{k+\sigma-1}$ is $s$-harmonic in $\Rp$ for any $k\in\{0,\ldots,2m+1\}$.}
	\end{align*}
\end{prop}
Observe that $(x_1)_+^{\sigma-1}$ is \textit{large} or \emph{singular} at $\partial\Rp$ in the sense that it diverges as $x_1\to 0$.  This is also the case of $u$ as in Theorem \ref{exp:sol:thm} if $h_0\neq 0$. These large $s$-harmonic functions are a purely nonlocal effect which can also be observed in other domains \cite{B99,AJS16b,nicola,G15:2}. The proof of Proposition \ref{s-harmonic2} is done first in dimension one $N=1$ using several pointwise manipulations, and then extended to higher-dimensions using the following result.
\begin{lemma}\label{lem:ldh}
	Let $s>0$, $\beta\in(0,1)$, $k,N\in \N$, $k<N$, $V\subset\R^k$ open, $U:=V\times \R^{N-k}$, and functions 
	\begin{align*}
		u:\R^N\to \R,\ v:\R^k\to \R,\ \ \text{ satisfying }\ \ u(x_1,\ldots,x_N)=v(x_1,\ldots, x_k)\ \ \text{for all }(x_1,\ldots,x_N)\in\R^N.
	\end{align*}
	Then $v\in \cL^1_s(\R^k)\cap C^{2s+\beta}(V)$ if and only if $u\in\cL^1_s(\R^N)\cap C^{2s+\beta}(U)$; furthermore,
	\begin{align*}
		(-\Delta)^s u(x_1,\ldots,x_{N})=(-\Delta)^s v(x_1,\ldots,x_{k})\qquad \text{ for all }x\in U.
	\end{align*}
	In particular, $u$ is $s$-harmonic in $U$ if and only if $v$ is $s$-harmonic in $V$.
\end{lemma}

A direct calculation (see Lemma  \ref{lem:prop-edenhofer} below) shows that the $s$-harmonic functions $x_1^{k+\sigma-1}$ have the following relationship to the summands in \eqref{edenhofer}: let $\sigma\in(0,1)$, $m\in \N$, and $k\in\{0,\ldots,m\}$, then
\begin{align*}
	\int_{\partial\Rp}\ \frac{(x_1)_+^{s+m-k}}{|x-z|^{N+2(m-k)}}\ dz=\frac{\pi^{\frac{N-1}{2}}\Gamma(m-k+\frac{1}{2})}{\Gamma(\frac{N}{2}+m-k)}\ (x_1)_+^{k+\sigma-1}\qquad \text{for $x\in \Rp$}.
\end{align*}
In particular, this shows the strong influence that compactly supported data has on the growth or decay of solutions at infinity. The fact that our results are presented for compactly supported functions is not only for a better presentation, but also because general data yield a more complex problem and even for the Laplacian $s=1$ this is an active research topic, see for example \cite{ZDQ16}. To mention one difficulty, without compact support bounds such as \eqref{bound}, \eqref{es}, or \eqref{es1} may not hold, and our uniqueness argument (see Lemma \ref{lem:bvpN} below) cannot be applied; in fact, without growth assumptions uniqueness does not hold, by Proposition \ref{s-harmonic2}.

\medskip

One of the key ingredients in our proofs is the following point inversion transformation.  Let $v\in \R^N$, $c\in\R$, $\kappa:\R^N\setminus\{-v\}\to\R^N$, and $K_{s}:C^{\infty}_c(\R^N\setminus\{-v\})\to C^{\infty}_c(\R^N)$ be given by 
\begin{align}\label{kappa:K:def}
	\kappa x:=c\frac{x+v}{|x+v|^{2}}-v\qquad \text{ and }\qquad K_{s} u (x):=|x+v|^{2s-N}u(\kappa x)\qquad \text{ for }x\in\R^N\backslash\{-v\}.
\end{align}

The next proposition states that $K_{s}$ preserves $s$-harmonicity.
\begin{prop}\label{lem:sharm}
	For $s,c>0$, $v\in \R^N$ let $\kappa$ and $K_s$ as in \eqref{kappa:K:def}. Then, for $u\in C^{\infty}_c(\R^N\setminus\{-v\})$ and $x\in \R^N\setminus\{-v\}$,
	\begin{align}\label{kelvin-trafo}
		(-\Delta)^s (K_s u)(x)&=c^{2s}\frac{K_s((-\Delta)^su)(x)}{|x+v|^{4s}}, \quad\text{ \emph{i.e.}, }\quad (-\Delta)^s \Big(\frac{u\circ \kappa (x)}{|x+v|^{N-2s}}\Big)=c^{2s}\frac{(-\Delta)^su(\kappa x)}{|x+v|^{N+2s}}.
	\end{align}
	As a consequence, if $U\subset \R^N\setminus\{-v\}$ is an open set and $u\in \cL^1_s$, then $u$ is distributionally $s$-harmonic in $U$ if and only if $K_su$ is distributionally $s$-harmonic in $\kappa(U)$.
\end{prop}
Proposition \ref{lem:sharm} can be deduced from \cite[Lemma 3]{DG2016}, where covariance under M\"{o}bius transformations is studied using a unique continuation argument. Here we present a different proof of Proposition \ref{lem:sharm} based on induction, which could be of independent interest.

Observe that, if $c=1$ and $v=0$, then $\kappa$ is the usual Kelvin transform which maps $B\backslash\{0\}$ to $\R^N\backslash \overline{B}$ conformally and vice versa; whereas, if $c=2$ and $v=e_1$, then $\kappa$ maps $B$ to the half-space $\R^N_+$ and $\partial B\backslash\{-e_1\}$ to $\partial\Rp$. Proposition \ref{lem:sharm} was known for $s\in(0,1)$, see for example \cite{BZ06,FW12}, and the case $s\in\N$ is classical.

\medskip

The $K_s$ transformation allows us to establish a link between \eqref{eq:intro:main} and its equivalent problem on balls, for which solutions and representation formulas are known \cite{AJS17a}. However, the case of higher-order boundary Poisson kernels is rather delicate, since $K_s$ does not map directly the trace operator of the ball to that of the half-space; in particular, these kernels require subtle and intricate combinatorial identities, which are inherent to higher-order problems.

We remark that Poisson kernels for the half-space associated to a large class of (local) higher-order elliptic operators satisfying general boundary conditions were obtained in \cite{ADN59}. Due to their generality, the formulas from \cite{ADN59} are much less explicit than the ones presented in Theorem \ref{exp:sol:thm}.

\medskip

The paper is organized as follows.  In Section \ref{pre:sec} we collect some useful results concerning integration by parts, properties of the trace operator \eqref{trace}, distributional solutions, and the explicit formula for the Green function in a ball. Section \ref{poi:sec} is devoted to properties of point-inversion transformations and the proof of Proposition \ref{lem:sharm}. Finally, Section \ref{E:k:sec} contains the proofs of Theorems \ref{thm:Poisson}, \ref{dist:sol:l}, \ref{exp:sol:thm}, Proposition \ref{s-harmonic2}, and Lemma \ref{lem:ldh}; this section is divided in 4 subsections dedicated respectively to the nonlocal Poisson kernel, the Green function, one-dimensional $s$-harmonic functions, and the boundary Poisson kernels.

\subsection{Notation}

For $m\in \N_0$ and $U\subset \R^N$ open we write $C^{m,0}(U)$ to denote the space of $m$-times continuously differentiable functions in $U$ and, for $\sigma\in(0,1]$ and $s=m+\sigma$, we write $C^s(U):=C^{m,\sigma}(U)$ to denote the space of functions in $C^{m,0}(U)$ whose derivatives of order $m$ are (locally) $\sigma$-H\"older continuous in $U$ or (locally) Lipschitz continuous in $U$ if $\sigma=1$. We denote by $C^s(\overline{U})$ the set of functions $u\in C^s(U)$ such that 
\begin{align*}
	\|u\|_{C^s(U)}:=\sum_{|\alpha|\leq m}\|\partial^{\alpha}u\|_{L^{\infty}(U)}+\sum_{|\alpha|=m}\ \sup_{\substack{x,y\in U \\x\neq y}}\  \frac{|\partial^{\alpha}u(x)-\partial^{\alpha}u(y)|}{|x-y|^{\sigma}}<\infty.
\end{align*}
Moreover, for $s\in(0,\infty]$, $C^s_c(U):=\{u\in C^s(\R^N): \supp \ u\subset U\text{ is compact}\}$ and $C^s_0(U):=\{u\in C^s(\R^N): u= 0 \text{ on $\R^N\setminus U$}\}$, where $\supp\ u:=\overline{\{ x\in U\;:\; u(x)\neq 0\}}$ is the support of $u$.  
\medskip

We use $u_+:= \max\{u,0\}$ to denote the positive part of $u$ and $(x_1)^{\beta}_+$, $\beta\in\R$, to denote the function
\begin{equation*}
	(x_1)^{\beta}_+:=\left\{\begin{aligned}
		&x_1^{\beta},&& \quad \text{if $x_1>0$,}\\
		&0,&&\quad  \text{if $x_1\leq 0$,}
	\end{aligned}\right.
\end{equation*}
If $f: \R^N\times\R^N\to \R$ we write $(-\Delta)_x^s f(x,y)$ to denote derivatives with respect to $x$, whenever they exist in some appropriate sense. 

\medskip

For $s>0$ we let (see \cite{FW16,S07,AJS16a,AJS16b})
\begin{align}\label{cL1s}
	\cL^1_{s} :=\cL^1_s(\R^N):=\Big\{u\in L^1_{loc}(\R^N)\;:\; \int_{\R^N}\frac{|u(x)|}{1+|x|^{N+2s}}\ dx<\infty \Big\}.
\end{align}

Let $\omega_{N}=2\pi^{N/2}\Gamma(N/2)^{-1}$ denote the $(N-1)$-dimensional measure of the sphere in $\R^N$.  We also use the following constants.
\begin{align}\label{constants}
	k_{N,s}:=\frac{\Gamma(\frac{N}{2})}{\pi^\frac{N}{2}4^s\Gamma(s)^2},\qquad {\alpha_{l,i}:=\frac{(-1)^i2^{l-2i}\Gamma(\frac{N}{2}+l-i)}{\pi^{\frac{N}{2}}(l-2i)!\ i!}},\qquad \gamma_{N,\sigma}:=\frac{\Gamma(\frac{N}{2})}{\pi^{\frac{N}{2}}\Gamma(\sigma)\Gamma(1-\sigma)},
\end{align}
where $\Gamma$ is the standard \emph{Gamma function}. Furthermore, for a set $A\subset \R^N$ we let $|A|$ denote its $N$-dimensional Lebesgue measure. For $x\in \R^N$ we sometimes write $x=(x_1,x')$ with $x_1\in \R$ and $x'\in \R^{N-1}$ ($x'=0$ if $N=1$). We also use the standard multi-index notation: for $\alpha=(\alpha_1,\ldots, \alpha_N)\in \N^N$,  
\begin{align*}
	{|\alpha|=\sum_{j=1}^N\alpha_j},\qquad \alpha!=\alpha_1!\ldots\alpha_N!,\qquad \partial^\alpha u = \frac{\partial^{|\alpha|} u}{\partial x_1^{\alpha_1}\cdots\partial x_N^{\alpha_N}},\qquad\text{ and } \qquad  y^\alpha = \prod_{i=1}^N y_i^{\alpha_i}.
\end{align*}

We use $\delta_{j,k}$ to denote a Kronecker delta, that is, $\delta_{j,k}=1$ if $j=k$ and $\delta_{j,k}=0$ if $j\neq k$.  We often use the identity \cite[equation (16) page 10]{EMOT81} 
\begin{align}\label{4674354}
	\int_0^\infty \frac{t^x}{(1+t^2)^{y}}\ dt=\frac{\Gamma(\frac{x+1}{2})\Gamma(y-\frac{x+1}{2})}{2\Gamma(y)}
	\qquad \text{for $x,y\in\R$, $0<\frac{x+1}{2}<y$}.
\end{align}
The constant in \eqref{Ds:def} is a positive normalization constant given by 
\begin{align}\label{cNms:def}
	c_{N,s}:=\frac{4^{s}\Gamma(\frac{N}{2}+s)}{ \pi^{\frac{N}{2}}\Gamma(-s) }\Big(\sum_{k=1}^{m+1}(-1)^{k}{ \binom{2m+2}{m+1-k}} k^{2s}\Big)^{-1},
\end{align}
which allows the following relationship: if $\cF(f)$ denotes the Fourier transform of $f$, then
\begin{align}\label{f:rel}
	\cF((\Delta)^s u)(\xi)=|\xi|^{2s}\cF(u)\qquad \text{ for all }u\in C_c^\infty(\R^N),
\end{align}
see, for example, \cite[Theorem 1.9]{AJS17b}. Finally, in dimension one ($N=1$), the boundary integral is meant in the sense $\int_{\partial \R^1_+}f(y)\ dy=f(0)$.

\section{Preliminary results and definitions}\label{pre:sec}

\subsection{Integration by parts formula}

We use the following integration by parts formula from \cite{AJS17b}.

\begin{lemma}\label{ibyp}[Lemma 1.5 from \cite{AJS17b}]
	Let $m\in \N$, $\beta\in(0,1)$, $\sigma\in(0,1]$, $s=m+\sigma$, $U\subset\R^N$ open, and $u\in C^{2s+\beta}(U)\cap \cL^1_s$. Then 
	\begin{align*}
		\int_{\R^N} u(x)(-\Delta)^s \varphi(x)\ dx=\int_{\R^N} (-\Delta)^su(x) \varphi(x)\ dx\qquad \text{ for all }\varphi\in C^\infty_c(U).
	\end{align*}
\end{lemma}

\subsection{Boundary trace operator}

Let $\sigma\in(0,1]$, $k\in \N_0$, and $\Rp:=\{x\in \R^N\;:\; x_1>0\}$. We now show some frequently used properties of the trace operator $D^{k+\sigma-1}$ defined in \eqref{trace}.
Observe that, by definition, for any $j\in \N_0$,
\[
\R^N\ni x\mapsto u_j(x)=(x_1)_+^{j+\sigma-1}\quad \text{satisfies}\quad D^{k+\sigma-1}u_j(0,x')=\delta_{j,k} \quad\text{ for all $x'\in \R^{N-1}$}
\]
and the Leibniz rule 
\begin{align}\label{leibniz}
	D^{k}(fg)=\sum_{j=0}^{k}(D^{j}f)(D^{k-j}g)  
\end{align}
holds. Moreover, we have the following lemma. 
\begin{lemma}\label{trace:lem}
	Let $k\in\N_0$, $\sigma\in(0,1]$, and $u:\R^N\to\R$ such that $x_1\mapsto x_1^{1-\sigma}u(x_1,x')$ is $C^k([0,\infty))$ for some $x'\in \R^{N-1}$. If 
	\[
	D^{j+\sigma-1}u(0,x')=0 \quad \text{ for $j\in \{0,\ldots,k-1\}$,}
	\] 	
	then
	\[
	D^{k+\sigma-1}u(0,x')=\lim_{x_1\to 0^+} x_1^{1-k-\sigma}u(x_1,x').
	\]
\end{lemma}
\begin{proof}
	Fix $x'\in \R^{N-1}$ and let $f:\R\to \R$, $f(r)=r^{1-\sigma}u(r,x')$. By assumption, $f\in C^{k}([0,\infty))$ and
	\begin{align*}
		0=D^{j}u(0,x')=\frac{1}{j!}\lim_{r\to 0^+} f^{(j)}(r)=\frac{1}{j!}f^{(j)}(0)\qquad \text{ for $j\in\{0,\ldots,k-1\}$}.
	\end{align*}
	Hence, by Taylor's theorem
	\[
	f(r)=\sum_{n=0}^{k}\frac{f^{(n)}(0)}{n!}r^n+o(r^k)=\frac{f^{(k)}(0)}{k!}r^k+o(r^k) \quad\text{as $r\to 0^+$,}
	\]
	and then
	\[
	D^{k+\sigma-1}u(0,x')=\frac{1}{k!}\lim_{r\to0^+}f^{(k)}(r)=\lim_{r\to0^+} f(r)r^{-k}=\lim_{r\to 0^+}\frac{u(r,x')}{r^{\sigma-1+k}},
	\]
	as claimed.
\end{proof}

\subsection{Distributional solutions}

Let $U\subset \R^N$ be an open set.  
We say that $u \in \cL^1_{s}$ is a \textit{distributional solution} of $(-\Delta)^su=f$ in $U$ if 
\begin{equation*}
	\int_{\R^N} u(x) (-\Delta)^s \varphi(x)\ dx  =  \int_U f(x) \varphi(x)\ dx   \qquad\text{ for all }\varphi\in C^{\infty}_c(U).
\end{equation*}
In particular, $u\in \cL^1_s$ is \textit{$s$-harmonic in the distributional sense in $U$} if 
\begin{align}\label{u:dist:sol}
	\int_{\R^N} u(x)(-\Delta)^s\varphi(x)\ dx =0\qquad \text{for all }\varphi\in C^{\infty}_c(U). 
\end{align}

For sufficiently smooth functions, $u$ is pointwisely $s$-harmonic if and only if $u$ is distributionally $s$-harmonic.
\begin{lemma}\label{ptod}
	Let $U\subset \R^N$ be an open set, $\beta>0$, and let $u\in \cL^1_s\cap C^{2s+\beta}(U)$.  Then $(-\Delta)^s u(x)=0$ for every $x\in U$ if and only if \eqref{u:dist:sol} holds.
\end{lemma}
\begin{proof}
	The result follows from Lemma \ref{ibyp} and the fundamental Lemma of calculus of variations.
\end{proof}

\subsection{Green functions}

Let $s>0$ and $\Omega\subset \R^N$ be a smooth open set.  A function $\cG_{\Omega}:\R^N\times \R^N\to\R\cup\{\infty\}$ is called a \textit{Green function of $(-\Delta)^s$ in $\Omega$}, if  $\cG_{\Omega}(x,y)=0$ for $x$ or $y$ in $\R^N\setminus \Omega$ and
\begin{align}\label{Green:prop}
	\int_{\R^N} \cG_{\Omega}(x,y)(-\Delta)^s\varphi(y)\ dy = \varphi(x) \qquad \text{ for all } x\in \Omega \text{ and } \varphi\in C^{\infty}_c(\Omega).
\end{align}
The Green function of $(-\Delta)^s$ in the unitary ball $B$ (see \cite{DG2016,AJS16b}) is given by
\begin{align}\label{green-ball1}
	\cG_B(x,y)&=k_{N,s}|x-y|^{2s-N}\int_0^{\frac{(1-|x|^2)_+(1-|y|^2)_+}{|x-y|^{2}}}\frac{v^{s-1}}{(v+1)^{\frac{N}{2}}}\ dv, \quad x,y\in \R^N,\ x\neq y,
\end{align}
where $k_{N,s}$ is as in \eqref{constants}.

\section{Point inversions}\label{poi:sec}

\begin{lemma}
	Let $N\in\N$, $s>0$, $u\in C^{\infty}_c(\R^N\setminus\{0\})$,  and let $K_s$ as in \eqref{kappa:K:def} with $c=1$ and  $v=0$. Then, for $x\in \R^N\setminus\{0\}$,
	\begin{align}
		x\cdot \n K_s(u)&=(2s-N)K_s(u)- K_s(x\cdot \n u),\label{id1} \\
		\Ds (|x|^2 \D u)&=-|x|^2 (-\D)^{s+1} u+ 4s \Ds( x\cdot \n u)+2s(N-2-2s) \Ds u, \label{id2}\\
		(-\Delta)^{s+1}u(x)&=(-\Delta)^{s}(-\Delta)u(x).\label{ch}
	\end{align}
\end{lemma}
\begin{proof}
	Identity \eqref{id1} follows from differentiating the function  
	\begin{align*}
		(0,\infty)\ni t\mapsto K_s(u)(tx) = t^{2s-N} |x|^{2s-N} u\left(\frac{x}{t|x|^{2}}\right) 
	\end{align*}
	with respect to $t$ and evaluating at $t=1$. Identity \eqref{ch} follows from \eqref{f:rel}. To show \eqref{id2} we use the Fourier transform of $u$, denoted by $\cF(u)$. First, let $\xi\in\R^N\backslash\{0\}$ and note that
	\begin{align}
		|\xi|^{2(1+s)}\D \cF(u) &= \D (\ |\xi|^{2(1+s)} \cF(u)\ )- 4(1+s)|\xi|^{2s} \xi \cdot\n \cF(u) - 2(1+s)(N+2s)|\xi|^{2s}\cF(u),\label{m1}\\
		\xi \cdot \nabla \cF(u)&=\sum_{j=1}^N \xi_j \frac{\partial}{\partial \xi_j}\cF(u) =\sum_{j=1}^N \xi_j \cF(-ix_j\, u) = i^2\cF(\div(xu))=-N\cF(u)-\cF(x\cdot \n u).\ \ \ \ \ \label{m2}
	\end{align}
	Then
	\begin{align*}
		\cF(\Ds (|x|^2 \D u))&=|\xi|^{2s} \cF(|x|^2 \D u)=|\xi|^{2s}(-\D) \cF (\D u)=|\xi|^{2s} \D (\ |\xi|^2 \cF(u)\ )\\
		&=|\xi|^{2s} \left( 2N \cF(u)+ 4 \xi \cdot\n  \cF(u)+|\xi|^2 \D \cF(u)  \right)\\
		&= |\xi|^{2(1+s)} \D \cF(u) + 4 |\xi|^{2s}\xi \cdot \n \cF(u)+ 2N|\xi|^{2s} \cF(u)\\
		&\stackrel{\eqref{m1}}{=}\D(\ |\xi|^{2(1+s)}  \cF(u)\ )-4s |\xi|^{2s}\xi \cdot \nabla \cF(u)-2s(2+N+2s)|\xi|^{2s} \cF(u)\\
		&\stackrel{\eqref{m2}}{=}\D(\ |\xi|^{2(1+s)}  \cF(u)\ )+4s |\xi|^{2s}\cF(x\cdot \n u)+2s(N-2-2s)|\xi|^{2s}\cF(u)
	\end{align*}
	and \eqref{id2} follows by applying the inverse Fourier transform.
\end{proof}

\begin{prop}\label{kelvin}
	For $s,c>0$, $v\in \R^N$ let $\kappa$ and $K_s$ as in \eqref{kappa:K:def}. Then, for $u\in C^{\infty}_c(\R^N\setminus\{-v\})$ and $x\in \R^N\setminus\{-v\}$,
	\begin{align}\label{kelvin-trafo:1}
		(-\Delta)^s (K_s u)(x)&=c^{2s}\frac{K_s((-\Delta)^su)(x)}{|x+v|^{4s}}, \quad\text{ \emph{i.e.}, }\quad (-\Delta)^s \Big(\frac{u\circ \kappa (x)}{|x+v|^{N-2s}}\Big)=c^{2s}\frac{(-\Delta)^su(\kappa x)}{|x+v|^{N+2s}}.
	\end{align}
\end{prop}
\begin{proof}
	Since 
	\begin{align*}
		(-\Delta)^s[u(cx+v)]=c^{2s}[(-\Delta)^su](cx+v) \quad\text{ for $s>0$, $c\in \R$, $x,v\in \R^N$, and $u\in C^{\infty}_c(\R^N)$,}
	\end{align*}
	we assume without loss of generality that $v=0$ and $c=1$. We show \eqref{kelvin-trafo:1} by induction.  For $s\in (0,1)$ the claim is known (see, for example, \cite[Proposition A.1]{RS12}) and for $s=1$ the claim follows from a direct computation.  Let $u\in C^{\infty}_c(\R^N\backslash \{0\})$, $s>0$, and assume as induction hypothesis that
	\begin{align}\label{ih}
		(-\D)^{s}  K_{s}(u )& = |x|^{-4s}K_{s}( (-\D)^{s}  u))\qquad \text{ in }\R^N\backslash\{0\}.
	\end{align}
	By \eqref{id1}, \eqref{ch}, \eqref{ih}, and since
	\begin{align*}
		\D K_1(u)=|x|^{-4} K_1(\D u) =|x|^{-2} K_1(|x|^2\D u)\qquad \text{ in }\R^N\backslash\{0\},
	\end{align*}
	it follows that, for $x\in\R^N\backslash\{0\}$,
	\begin{align*}
		& (-\D)^{s+1} K_{s+1}(u )= - \Ds ( \D( |x|^{2s} K_1(u)))\\
		&=-\Ds \Bigl(  2s (N+2s-2)|x|^{2s-2}K_1(u)+ |x|^{2s}\D K_1(u)+4s |x|^{2s-2}x\cdot \n K_1(u)     \Bigr) \\
		&=-\Ds  \Bigl(  2s (-N+2s+2)|x|^{2s-2}K_1(u)+ |x|^{2s}\D K_1(u)-4s |x|^{2s-2}  K_1(x\cdot \n u)      \Bigr)\\
		&=-\Ds  \Bigl(   2s (-N+2s+2) K_s(u)+ |x|^{2s}\D K_1(u)-4s   K_s(x\cdot \n u)      \Bigr)\\
		& =-\Ds \Bigl(  2s (-N+2s+2) K_s(u)+ K_s(|x|^2\D u)-4s   K_s(x\cdot \n u)     \Bigr)\\
		& =   2s (N-2-2s)|x|^{-4s} K_s( \Ds  u)- |x|^{-4s}K_s( \Ds (|x|^2\D u))+4s |x|^{4s}  K_s(\Ds (x\cdot \n u)).
	\end{align*}
	Using  \eqref{id2} and the linearity of $K_s$, we deduce that 
	\begin{align*}
		(-\D)^{s+1}  K_{s+1}(u )& =  |x|^{-4s}K_s(|x|^2(-\D)^{s+1}  u))= |x|^{-4(s+1)}K_{s+1}( (-\D)^{s+1}  u))\qquad \text{ in }\R^N\backslash\{0\},
	\end{align*}
	as claimed.
\end{proof}

\begin{lemma}\label{KsinL1s}
	Let $c>0$, $v\in \R^N$. If $u\in\cL^1_s$ then $K_s u\in\cL^1_s$.
\end{lemma}
\begin{proof}
	Recall that the (absolute value of the) Jacobian for $x\mapsto c\frac{x+v}{|x+v|^{2}}-v$ is $c^N|x+v|^{-2N}$. Let $k$ denote possibly different constants, then
	\begin{align*}
		\int_{\R^N}\frac{|K_s u(x)|}{1+|x|^{2s+N}}\ dx
		&=k\int_{\R^N}\frac{|x+v|^{2s-N}|u(\kappa x)|}{1+|x|^{2s+N}}\ dx
		=k\int_{\R^N}|x+v|^{-2N}\frac{|\kappa x+v|^{2s-N}|u(x)|}{1+|\kappa x|^{2s+N}}\ dx\\
		&=k\int_{\R^N}|x+v|^{-2N}\frac{|x+v|^{N-2s}|u(x)|}{1+|c\frac{x+v}{|x+v|^{2}}-v|^{2s+N}}\ dx\\
		&=k\int_{\R^N}\frac{|u(x)|}{|x+v|^{N+2s}+|c\frac{x+v}{|x+v|}-|x+v|v|^{2s+N}}\ dx\\
		&=k\int_{\R^N}\frac{|u(x-v)|}{|x|^{N+2s}+|c\frac{x}{|x|}-|x|v|^{2s+N}}\ dx<\infty.
	\end{align*}
\end{proof}

\begin{proof}[Proof of Proposition \ref{lem:sharm}]
	Formula \eqref{kelvin-trafo} follows from Proposition \ref{kelvin}.  
	Let $U\subset \R^N\setminus\{-v\}$ be an open set, $u\in \cL^1_s$, and $\phi\in C^\infty_c(U)$. Note that $K_s u\in \cL^1_s$, by Lemma \ref{KsinL1s}; furthermore, $K_s\phi\in C_c^\infty(\kappa(U))$, $K_s(K_s u)=c^{2s-N}u$, and $K_s(K_s\phi)=c^{2s-N}\phi$. Then, by Proposition \ref{kelvin},
	\begin{align*}
		& \int_{\R^N}u(x)\,\Ds\phi(x)\ dx\ =\ c^{2N-4s}\int_{\R^N}K_s(K_su)(x)\,\Ds K_s(K_s\phi)(x)\ dx \\
		& =c^{2N-2s}\int_{\R^N}|x+v|^{-2N}\,K_su(\kappa x)\,\Ds K_s\phi(\kappa x)\ dx\ 
		=\ c^{N-2s}\int_{\R^N}\,K_su(y)\,\Ds K_s\phi(y)\ dy.
	\end{align*}
	As a consequence, $u$ is distributionally $s$-harmonic in $U$ if and only if $K_su$ is distributionally $s$-harmonic in $\kappa(U)$, as claimed.
\end{proof}

\subsection{Uniqueness}\label{rep:sec}

Now we can use $K_s$ (as in \eqref{kappa:K:def} with $c=2$ and $v=e_1$) to establish uniqueness of solutions to homogeneous problems in $\Rp$.

\begin{lemma}\label{lem:bvpN}
	Let $N\in \N$, $m\in\N_0$, $\beta,\sigma\in(0,1)$, and $s=m+\sigma$. Let $u\in C^{2s+\beta}(\Rp)\cap\cL^1_s$ satisfy $(-\Delta)^su=0$ in $\Rp$, $u=0$ on $\R^N\setminus \Rp$, and assume there is $C>0$ and $\alpha\in(0,1)$ such that
	\begin{align*}
		|u(x)|\leq C \frac{\max\{x_1^{s-1+\alpha},x_1^s\}}{1+|x|^{N}}\qquad \text{ for }x\in\Rp.
	\end{align*}
	Then $D^{k+\sigma-1}u=0$ in $\partial \Rp$ for $k=0,\ldots,m$ and $u\equiv0$ on $\R^N$.
\end{lemma}
\begin{proof}
	Let $u$ be as stated and note that $w=K_su$ is in $C^{2s+\beta}(B)\cap \cL^1_s$, by Lemma \ref{KsinL1s} (here $K_s$ is as in \eqref{kappa:K:def} with $c=2$ and $v=e_1$).  For $x\in\overline{B}$ put $\delta(x):=(1-|x|^2)$. 
	Note that $(\kappa x)_1 = \frac{\delta(x)}{|x+e_1|^2}$ for all $x\in\R^N\backslash \{-e_1\}$. Moreover, by adjusting the constant $C$, we also have
	\[
	|u(x)|\leq C \frac{\max\{x_1^{s-1+\alpha},x_1^s\}}{1+|x+e_1|^{N}}\qquad \text{ for }x\in\Rp.
	\]
	Then, using the monotonicity of the Kelvin transform, we have for $x\in\overline{B}\backslash\{-e_1\}$,
	\begin{align*}
		|w(x)|&\leq C |x+e_1|^{2s-N}\frac{
			\max\{|(\kappa x)_1|^{s-1+\alpha},|(\kappa x)_1|^s\}
		}{1+|\kappa x +e_1|^N} \\
		&= C \frac{ \max\{\delta(x)^{s-1+\alpha}|x+e_1|^{2-N-2\alpha},|x+e_1|^{-N}\delta(x)^s\}}{1+2^N|x+e_1|^{-N}}\\
		& =C \frac{ \max\{\delta(x)^{s-1+\alpha}|x+e_1|^{2-2\alpha},\delta(x)^s\}}{|x+e_1|^{N}+2^N}
		\leq \frac{C}{2^{N}} \max\{\delta(x)^{s-1+\alpha},\delta(x)^s\}=\frac{C}{2^{N}}\delta(x)^{s-1+\alpha}.
	\end{align*}
	Since $w$ is $s$-harmonic in $B$, by Proposition \ref{lem:sharm}, and satisfies that $w=0$ in $\R^N\backslash \overline{B}$ and $|w(x)|\leq C\delta(x)^{s-1+\alpha}$ for $x\in\overline{B}$, we have that $w\equiv 0$, by \cite[Theorem 1.5]{AJS17a}. Then $u(x)=2^{N-2s}K_s(w)(x)=0$ for $x\in \R^N\backslash\{-e_1\}$, and the statement follows.
\end{proof}

\section{Explicit kernels}\label{E:k:sec}

\subsection{The nonlocal Poisson kernel}

\begin{proof}[Proof of Theorem \ref{thm:Poisson}]
	Let $g\in L^1(\R^N)$ with compact support in $\R^N\backslash \overline{\Rp}$, $B:=B_1(0)$, $K_s$ and $\kappa$ be given by \eqref{kappa:K:def} with $v=e_1$, $c=2$, and let $\gamma_{N,\sigma}$ as in \eqref{constants}. Note that $K_sg\in \cL^1_s$, by Lemma \ref{KsinL1s}, and $K_sg=0$ on $\R^N\setminus \overline{B_{1+\tilde r}(0)}$ for some $\tilde r\in(0,1)$. Hence, by \cite[Theorem 1.1]{AJS17a}, the function $v:\R^N\to \R$ given by
	\begin{align}\label{v:def}
		v(x)=(-1)^m\gamma_{N,\sigma}2^{N-2s}\int_{\R^N\setminus B} \frac{(1-|x|^2)_+^s}{(|y|^2-1)^s|x-y|^{N}}K_sg(y)\ dy+2^{N-2s}K_sg(x)\chi_{\R^N\setminus B}(x)
	\end{align}
	satisfies that $v\in C^{\infty}(B)\cap C^{s}_0(B)\cap \cL^1_s$ and $(-\Delta)^s v(x)=0$ for every $x\in B=\kappa(\Rp)$.
	
	Recall that the absolute value of the Jacobian for $y\mapsto \kappa y$ is $2^N|x+v|^{-2N}$ and, for $x,y\in \R^N\setminus\{-e_1\}$,
	\begin{align}\label{he}
		|\kappa x-\kappa y|=\frac{2|x-y|}{|x+e_1||y+e_1|}, \qquad 1-|\kappa x|^2=2^{2}\frac{x_1}{|x+e_1|^{2}},\qquad 
		|\kappa y+e_1|=\frac{2}{|y+e_1|}.
	\end{align}
	We claim that $u=K_s v$ in $\R^N\backslash\{-e_1\}$.  Indeed, for $x\in \R^N\setminus \Rp=\kappa(\R^N\setminus B)$ we have that $K_s v(x)=2^{N-2s}K_s(K_sg)(x)=2^{N-2s}2^{2s-N}g(x)=u(x)$. Then, using \eqref{he}, we have for $x\in \Rp{\backslash \{-e_1\}}$ that
	\begin{align*}
		K_s&(v)(x)=(-1)^m\gamma_{N,\sigma}2^{N-2s}|x+e_1|^{2s-N}\int_{\R^N\setminus B} \frac{(1-|\kappa x|^2)^s}{(|y|^2-1)^s|\kappa x-y|^{N}}K_sg(y)\ dy\\
		&=(-1)^m\gamma_{N,\sigma}2^{N-2s}|x+e_1|^{2s-N}\frac{{2^{2s}}x_1^s}{|x+e_1|^{2s}}\int_{\R^N\setminus B} \frac{|y+e_1|^{2s-N}}{(|y|^2-1)^s|\kappa x-y|^{N}}g(\kappa y)\ dy\\
		&=(-1)^m\gamma_{N,\sigma}2^{2N}|x+e_1|^{-N}x_1^s\int_{\R^N\setminus  \Rp} \frac{|\kappa y+e_1|^{2s-N}}{(|\kappa y|^2-1)^s|\kappa x-\kappa y|^{N}}|y+e_1|^{-2N}g(y)\ dy\\
		&=(-1)^m\gamma_{N,\sigma}2^{2N}|x+e_1|^{-N}x_1^s\int_{\R^N\setminus \Rp} \frac{{2^{2s-N}|y+e_1|^{N-2s}|y+e_1|^{2s}|x+e_1|^N|y+e_1|^N}}{2^{2s}(-y_1)^s{2^N}|x-y|^{N}| y+e_1|^{2N}}g(y)\ dy\\
		&=(-1)^m\gamma_{N,\sigma}\int_{\R^N\setminus \Rp} \frac{x_1^s}{(-y_1)^s|x-y|^{N}}g(y)\ dy=u(x).
	\end{align*}
	Therefore $u=K_s v$ in $\R^N\backslash\{-e_1\}$. Since $\kappa:\R^N\backslash\{-e_1\}\to \R^N$ is a smooth transformation we have that $u=K_s v\in {C^{\infty}}(\Rp)$ and, by Lemma \ref{KsinL1s}, $u=K_s v\in \cL^1_s$. By Proposition \ref{lem:sharm} and Lemma \ref{ptod}, it follows that $(-\Delta)^s u(x)=0$ for every $x\in \Rp$.  Finally, to see that $u\in C^{s}_0(\Rp)$ and \eqref{bound}, note that $|v(x)|\leq C (1-|x|^2)^s$ for $x\in B$ and for some $C>0$ (depending on $g$, $N$, and $s$); but then, by linearity,
	\begin{align}\label{st}
		0\leq K_s(C (1-|x|^2)^s-|v(x)|)=|x+e_1|^{2s-N}C\frac{4^{s}x_1^s}{|x+e_1|^{2s}}-|u(x)|\qquad \text{ for }x\in \Rp
	\end{align}
	and thus $|u(x)|\leq C4^{s}|x+e_1|^{-N}x_1^s$ for $x\in\Rp$ which implies \eqref{bound}. Note that {$u$} is the unique solution of \eqref{eq0} satisfying \eqref{bound}, by Lemma \ref{lem:bvpN}.
	
	We now argue \eqref{cor:bound}. Assume in addition that $g$ is nonnegative and let
	$u=K_s v$ in $\R^N\backslash\{-e_1\}$ where $v$ is given by \eqref{v:def}, which implies that $(-1)^mu>0$ in $\Rp$, since $\gamma_{N,\sigma}>0$ by \eqref{constants} and $K_s$ preserves positivity. The upper bound in \eqref{cor:bound} follows form {\eqref{st}} and the lower bound can be argued similarly since, for $x\in B$,
	\begin{align*}
		(-1)^m v(x)=\gamma_{N,\sigma}2^{N-2s}(1-|x|^2)^s\int_{\R^N\setminus \overline{B_{1+\tilde r}}} \frac{K_sg(y)}{(|y|^2-1)^s|x-y|^{N}}\ dy \geq c (1-|x|^2)^s
	\end{align*}
	where $ c=\inf_{x\in \overline{B}}\int_{V} \frac{K_sg(y)}{(|y|^2-1)^s|x-y|^{N}}\ dy>0$ with $V:=\R^N\setminus \overline {B_{1+\tilde r}}$.
\end{proof}

\subsection{The Green function}\label{green:sec}

We now show an extension of \cite[Theorem 2]{BZ06} to general $s>0$.
\begin{prop}\label{kelvin-green}
	Let $s,c>0$, $v\in \R^N$, and $D\subset\R^N\setminus\{-v\}$ be an open set and denote $U:=\kappa(D)$. Let $\cG_D$ be a Green function of $(-\Delta)^s$ in $D$. Then $\cG_U:\R^N\times \R^N\to\R$ given by 
	\begin{align*}
		\cG_U(x,y)=\Big(\ \frac{c}{|\kappa x+v||\kappa y+v|}\ \Big)^{2s-N}\cG_D(\kappa x,\kappa y),\qquad x,y\in \R^N,\ x\neq y
	\end{align*}
	is a Green function of $(-\Delta)^s$ in $U$.
\end{prop}
\begin{proof}
	Let $\varphi\in C^{\infty}_c(U)$ and note that $K_s\varphi\in C^{\infty}_c(D)$. Note that $|\det J_\kappa(z)|=c^N|z+v|^{-2N}$ for $z\in \R^N\setminus\{-v\}$. Similarly, as in the proof of Proposition \ref{lem:sharm} it follows that
	\begin{align*}
		&\int_U\cG_U(x,y)(-\Delta)^s\varphi(y)\ dy
		=\int_D\cG_U(x,\kappa y)(-\Delta)^s\varphi(\kappa y)\frac{c^{N}}{|y+v|^{2N}}\ dy\\
		& =\int_D\cG_U(x,\kappa y)c^{2s}\frac{(-\Delta)^s\varphi(\kappa y)}{|y+v|^{N+2s}}\frac{c^{N-2s}}{|y+v|^{N-2s}}\ dy \\
		& =\int_D\cG_U(x,\kappa y) (-\Delta)^s \Big(\frac{\varphi\circ \kappa (y)}{|y+v|^{N-2s}}\Big) \frac{c^{N-2s}}{|y+v|^{N-2s}}\ dy\\
		&=\frac{1}{|\kappa x+v|^{2s-N}}\int_D  \cG_D(\kappa x,y) (-\Delta)^s \Big(\frac{\varphi\circ \kappa (y)}{|y+v|^{N-2s}}\Big)\ dy
		=\frac{1}{|\kappa x+v|^{2s-N}} \Big(\frac{\varphi\circ \kappa (\kappa x)}{|\kappa x+v|^{N-2s}}\Big)=\varphi(x),
	\end{align*}
	which shows that $\cG_U$ is a Green function for $(-\Delta)^s$ in $U$.
\end{proof}

\begin{cor}\label{green-half}
	For $s>0$ a Green function of $(-\Delta)^s$ in $\Rp:=\{x\in \R^N\;:\;x_1>0\}$ is given by \eqref{greenhs0}.
\end{cor}
\begin{proof}
	Let $\kappa= 2\frac{x+e_1}{|x+e_1|^{2}}-e_1$ (i.e. $c=2$, $v=e_1$), then $\kappa(B)=\Rp=\{x\in \R^N\;:\; x_1>0\}$ and the claim follows from Proposition \ref{kelvin-green} and a simple direct calculation.
\end{proof}

\begin{remark}
	Let $f,g\geq 0$ be functions defined on the same set $D$.
	We write $f\simeq g$ if there is $c>0$ such that $\frac1c g(x)\leq f(x)\leq cg(x)$ for all $x\in D$ and define $d(x):=1-|x|$. Recall the definition of the Green function of the ball $G_B$ given by \eqref{green-ball1}. In $\overline{B}\times \overline{B}$ we have
	\begin{align*}
		\cG_B(x,y)\simeq \left\{\begin{aligned}
			&|x-y|^{2s-N}\min\Big\{1,\frac{d(x)^sd(y)^s}{|x-y|^{2s}}\Big\},&& \quad \text{ if $N>2s$,}\\
			&\ln\Big(1+\frac{d(x)^sd(y)^s}{|x-y|^{2s}}\Big),&& \quad \text{ if $N=2s$,}\\
			&d(x)^{s-\frac{N}{2}}d(y)^{s-\frac{N}{2}}\min\left\{1,\frac{d(x)^{\frac{N}{2}}d(y)^{\frac{N}{2}}}{|x-y|^{N}}\right\},&& \quad \text{ if $N<2s$.}\\
		\end{aligned}\right.
	\end{align*}
	These types of estimates are known if $s\in\N\cup(0,1)$, see, for example, \cite{CS98,GGS10}. We refer to \cite[Theorem 4.6]{GGS10}, where the case $s\in\N$ is considered, but the same proof can be used for any $s>0$.  Using Proposition \ref{kelvin-green} and that, for $s,t\geq0$, $\lambda>0$, $x,y\in \Rp$,
	\begin{align*}
		\min\{s,t\}\simeq \frac{st}{s+t},\qquad |x-\bar y|^2-|x-y|^2=4x_ny_n,\qquad 
		\ln(1+t^\lambda) \simeq (\min\{1, t^\lambda\})\ln(2+t),
	\end{align*}
	where $\bar y:=(-y_1,\ldots,y_N)$ for $y=(y_1,\ldots,y_N)\in\R^N,$ one can deduce the estimates
	\begin{align}\label{Green:est}
		\cG_s(x,y)\simeq \left\{\begin{aligned}
			&\frac{(x_1y_1)^s}{|x-y|^{N-2s}|x-\bar y|^{2s}},&& \quad \text{ if $N>2s$,}\\
			&\frac{(x_1y_1)^s}{|x-\bar y|^{2s}}\ln\Big(1+\frac{|x-\bar y|^2}{|x-y|^{s}}\Big),&& \quad \text{ if $N=2s$,}\\
			&\frac{(x_1y_1)^s}{|x-y|^{N}},&& \quad \text{ if $N<2s$}\\
		\end{aligned}\right.
	\end{align}
	in $\Rp\times\Rp$, see for example \cite[Corollary 2.5]{BMZ04}.
\end{remark}

We now show Theorem \ref{dist:sol:l}.
\begin{proof}[Proof of Theorem \ref{dist:sol:l}]
	Let $\varphi\in C^\infty_c(\R^N)$, then, by Corollary \ref{green-half}, \eqref{Green:prop}, and the Fubini's theorem,
	\begin{align*}
		\int_{\Rp} u(x)\Ds\varphi(x)\ dx &
		= \int_{\Rp} \int_{\Rp} \cG_s(x,y)f(y)\ dy\Ds\varphi(x)\ dx\\
		&= \int_{\Rp} f(y)\int_{\Rp} \cG_s(x,y)\Ds\varphi(x)\ dx\ dy=\int_{\Rp} f(x)\varphi(x)\ dx.
	\end{align*}
	Therefore $u$ is a distributional solution of \eqref{eq}.  To see that $u\in C^{2s+\beta}(\Rp)$, let  $x\in\R^N\backslash\{-e_1\}$ and $\kappa x=2\frac{x+e_1}{|x+e_1|^2}-e_1$, then (see Proposition \ref{kelvin-green}) we have that
	\begin{align*}
		\cG_s(x,y) = \Big(\ \frac{2}{|\kappa x+e_1||\kappa y+e_1|}\ \Big)^{2s-N}\cG_B(\kappa x,\kappa y),\qquad x,y\in \R^N,\ x\neq y,
	\end{align*}
	where $\cG_B$ is the Green function of the ball. 
	Then, for $x\in \Rp$,
	\begin{align}
		u(x)&=\int_{\Rp}\Big(\ \frac{2}{|\kappa x+e_1||\kappa y+e_1|}\ \Big)^{2s-N}\cG_B(\kappa x,\kappa y)f(y)\ dy\nonumber\\
		&=\int_{B}\Big(\ \frac{2}{|\kappa x+e_1||y+e_1|}\ \Big)^{2s-N}\cG_B(\kappa x, y)f(\kappa y)2^{N}|y+e_1|^{-2N}\ dy\nonumber\\
		&=2^{2s}|\kappa x+e_1|^{N-2s}\int_{B}\cG_B(\kappa x, y)\frac{f(\kappa y)}{|y+e_1|^{N+2s}}\ dy.\label{equal}
	\end{align}
	Since $y\mapsto \frac{f(\kappa y)}{|y+e_1|^{N+2s}}$ is a function in $C^\beta_c(B)$, we have by \cite[Theorem 1.1]{AJS16b} (see also \cite[Theorem 1]{DG2016}) that
	\begin{align*}
		x\mapsto \int_{B}\cG_B(x, y)\frac{f(\kappa y)}{|y+e_1|^{N+2s}}\ dy\in C^{2s+\beta'}(B) \quad\text{ for $\beta'<\beta$ ($\beta'=\beta$ is allowed, if $2s+\beta\notin\N$),}
	\end{align*}
	which implies, by \eqref{equal}, that $u\in C^{2s+\beta'}(\Rp)$ , $\beta'<\beta$ (where again $\beta'=\beta$ is allowed, if $2s+\beta\notin\N$). Finally, we argue the estimates \eqref{es}, which follow from \eqref{Green:est} and the fact that $f\in C^\infty_c(\Rp)$: for instance, let $N<2s$, $K:=\operatorname{supp}(f)\subset \Rp$, $d:=\operatorname{dist}(\partial \Rp,K)/2$,  then, by \eqref{Green:est}, there are $C_1(N)=C_1>0$ and $C(K)=C>0$ such that
	\begin{align*}
		\cG_s(x,y)\leq C_1\frac{(x_1y_1)^s}{|x-y|^{N}} \leq C \frac{x_1^s}{1+|x|^{N}}\qquad \text{ for all }y\in K,\ x\in A:=\{z\in \Rp\::\: \operatorname{dist}(z,K)>d\}.
	\end{align*}
	Note also that $\Rp\backslash A\subset  \Rp$ is compact and that $u$ is continuous (and therefore bounded) on $\overline{\Rp\backslash A}$. Then there is $C'(K)=C'>0$ such that
	\begin{align*}
		|u(x)|\leq \int_{\Rp} \cG_s(x,y)|f(y)|\ dy\leq C'\int_{K}|f(y)|\ dy\ \frac{x_1^s}{1+|x|^{N}}\qquad \text{for all }x\in\Rp.
	\end{align*}
	The cases $N=2s$ and $N>2s$ follow similarly, and thus \eqref{es} holds. Finally, the uniqueness follows from Lemma \ref{lem:bvpN}.
\end{proof}

\begin{remark}\label{remark:crucial}
	In Theorem \ref{dist:sol:l}, if $f\in C_c^\beta(\Rp)$ and $2s+\beta\in \N$, then $u\in C^{2s+\beta'}(B)$ for $\beta'<\beta$ (by Theorem \ref{dist:sol:l} and because $C_c^{\beta}(\Rp)\subset C_c^{\beta'}(\Rp)$), but one cannot guarantee, in general, that $u\in C^{2s+\beta}(\Rp)$, see e.g. \cite{G15:2} or the recent work \cite{GKL18}, where some counterexamples are shown.
\end{remark}

\subsection{One-dimensional \texorpdfstring{$s$}{s}-harmonic functions}\label{sharm:sec}

In this subsection we show Lemma \ref{lem:ldh} and Proposition~\ref{s-harmonic2}.
We mention that a similar result to Lemma~\ref{lem:ldh} is remarked in \cite[proof of Theorem 3]{PSV13} for $s\in(0,1)$.
\begin{proof}[Proof of Lemma \ref{lem:ldh}]
	We show first that 
	\begin{align}\label{equivs}
		\text{$v\in \cL^1_s(\R^k)\cap C^{2s+\beta}(V)$ if and only if $u\in\cL^1_s(\R^N)\cap C^{2s+\beta}(U)$.}
	\end{align}
	Let $A:=B_1(0)\subset \R^k$ and $B:=B_1(0)\subset \R^{N-k}$, then
	\begin{align*}
		\int_{\R^N\setminus A\times B}&\frac{|u(x)|}{1+|x|^{N+2s}}\ dx 
		\leq \int_{\R^k\setminus A}|v(y)| \omega_{N-k}\int_{1}^{\infty} \frac{r^{N-k-1}}{(|y|^2+r^2)^{\frac{N}{2}+s}}\ dr\ dy\\
		&\leq \int_{\R^k\setminus A}\frac{|v(y)|}{|y|^{N+2s}} \omega_{N-k}\int_{1}^{\infty} \frac{r^{N-k-1}}{(1+(r/|y|)^2)^{\frac{N}{2}+s}}\ dr\ dy
		\leq  \omega_{N-k}\int_{0}^{\infty} \frac{t^{N-k-1}}{(1+t^2)^{\frac{N}{2}+s}}\ dt \int_{\R^k\setminus A}\frac{|v(y)|}{|y|^{k+2s}}\ dy.
	\end{align*}
	and
	\begin{align*}
		&\int_{\R^N\setminus A\times B}\frac{|u(x)|}{1+|x|^{N+2s}}\ dx = \int_{\R^k\setminus A}|v(y)| \omega_{N-k}\int_{1}^{\infty} \frac{r^{N-k-1}}{1+(|y|^2+r^2)^{\frac{N}{2}+s}}\ dr\ dy\\
		&\quad= \int_{\R^k\setminus A}\frac{|v(y)|}{|y|^{k+2s}} \omega_{N-k}\int_{1}^{\infty} \frac{r^{N-k-1}}{\frac{1}{|y|^{N+2s}}+(1+r^2)^{\frac{N}{2}+s}}\ dr\ dy
		\geq \omega_{N-k}\int_{1}^{\infty} \frac{r^{N-k-1}}{1+(1+r^2)^{\frac{N}{2}+s}}\ dr\int_{\R^k\setminus A}\frac{|v(y)|}{|y|^{k+2s}}\ dy.
	\end{align*}
	Since $\int_{1}^{\infty} \frac{r^{N-k-1}}{1+(1+r^2)^{\frac{N}{2}+s}}\ dr<\infty$ and $\int_{0}^{\infty} \frac{t^{N-k-1}}{(1+t^2)^{\frac{N}{2}+s}}\ dt<\infty$ we have that
	\[
	\int_{\R^N\setminus A\times B}\frac{|u(x)|}{1+|x|^{N+2s}}\ dx<\infty \quad\text{if and only if}\quad \int_{\R^k\setminus A}\frac{|v(y)|}{|y|^{k+2s}}\ dy <\infty.
	\]
	Since $v\in L^1_{loc}(\R^k)$ if and only if $u\in L^1_{loc}(\R^N)$, by Fubini's theorem, we obtain that \eqref{equivs} holds. Next, fix  $x=(x',x'')\in U\subset \R^N$ for $x'\in V\subset \R^k$ and $x''\in \R^{N-k}$, then, by a change of variables,
	\begin{align*}
		(-\Delta)^s u(x)
		&=\frac{c_{N,s}}{2}\int_{\R^N} \frac{\sum_{k=-{m-1}}^{m+1} (-1)^k { \binom{2m+2}{m+1-k}} u(x+ky)}{|y|^{N+2s}} \ dy\\
		&=\frac{c_{N,s}}{2}\int_{\R^k}\sum_{k=-{m-1}}^{m+1} (-1)^k { \binom{2m+2}{m+1-k}} v(x'+ky') \int_{\R^{N-k}}\frac{1}{(|y'|^2+|y''|^2)^{\frac{N}{2}+s}} \ dy'' dy'\\
		&=\frac{c_{N,s}}{2}\int_{\R^k}\frac{\sum_{k=-{m-1}}^{m+1} (-1)^k { \binom{2m+2}{m+1-k}} v(x'+ky')}{|y'|^{k+2s}} \int_{\R^{N-k}}\frac{1}{(1+|z''|^2)^{\frac{N}{2}+s}} \ dz'' dy'\\
		&=\frac{c_{N,s}}{2}\int_{\R^{N-k}}\frac{1}{(1+|z''|^2)^{\frac{N}{2}+s}} \ dz''\frac{2}{c_{k,s}}(-\Delta)^s v(x')=(-\Delta)^s v(x'),
	\end{align*}
	since, by \eqref{cNms:def} and \eqref{4674354},
	\begin{align*}
		\frac{c_{N,s}}{c_{k,s}}\int_{\R^{N-k}}(1+|t|^2)^{-\frac{N}{2}-s} \ dt
		& =\frac{\Gamma(\frac{N}{2}+s)}{ \Gamma(\frac{k}{2}+s)\pi^{\frac{N-k}{2}}}
		\omega_{N-k}\int_{0}^\infty(1+|r|^2)^{-\frac{N}{2}-s}r^{N-k-1} \ dr\\
		& =\frac{\Gamma(\frac{N}{2}+s)}{ \Gamma(\frac{k}{2}+s)\pi^{\frac{N-k}{2}}}\ 
		\frac{2\pi^{\frac{N-k}{2}}}{\Gamma(\frac{N-k}{2})}
		\ \frac{\Gamma \left(\frac{N-k}{2}\right) \Gamma \left(\frac{k}{2}+s\right)}{2 \Gamma
			\left(\frac{N}{2}+s\right)}
		=1.
	\end{align*}
\end{proof}

\begin{remark}
	Under the weaker assumptions $v\in \cL^1_s(\R^k)$ and $u\in \cL^1_s(\R^N)$ one can also show that $(-\Delta)^s u(x_1,\ldots,x_N)=(-\Delta)^s v(x_1,\ldots,x_k)$ holds in a distributional sense.
\end{remark}

Using Lemma \ref{lem:ldh} we next show Proposition \ref{s-harmonic2} for $m=0$. We note that the $\sigma$-harmonicity of $x\mapsto (x_1)_+^{\sigma}$ is observed in \cite[equation (3.39)]{B99}, but for completeness we give here a proof.
\begin{lemma}\label{s-harmonic1}
	For $\sigma\in(0,1)$ the functions $x\mapsto (x_1)_+^{\sigma-1}$ and $x\mapsto (x_1)_+^{\sigma}$ are pointwisely $\sigma$-har\-mon\-ic in $\Rp$.
\end{lemma}
\begin{proof}
	The function ${x\mapsto}\frac{(1-|x|^2)_+^s}{|x+e_1|^N}$ for $x\in\R^N$ is pointwisely $s$-harmonic in $B$ by \cite[Proposition 1.2]{AJS16b} or by \cite[Corollary 1.7]{AJS17a}.  Let $K_\sigma$ and $\kappa$ as in \eqref{kappa:K:def} with $c=2$ and $v=e_1$.  Then $\kappa(B)=\R^N_+$ and, for $x\in \R^N$, 
	\begin{align*}
		K_\sigma\Big(\frac{(1-|x|^2)_+^\sigma}{|x+e_1|^N}\Big)
		&=|x+e_1|^{2\sigma-N}\frac{(1-|\kappa x|^2)_+^\sigma}{|\kappa x+e_1|^N}\nonumber\\
		&=|x+e_1|^{2\sigma-N}(\frac{4(x_1)_+}{|x+e_1|^2})^\sigma(\frac{2}{|x+e_1|})^{-N}=2^{2\sigma-N}(x_1)_+^\sigma,
	\end{align*}
	and $x\mapsto (x_1)_+^{\sigma}$ is $\sigma$-harmonic in $\Rp$, by Proposition \ref{lem:sharm} and Lemma \ref{ptod}.
	
	We now show that $w:\R^N\to\R$ given by $w(x):=(x_1)_+^{\sigma-1}$ is $\sigma$-harmonic in $\Rp$.  By Lemma \ref{lem:ldh} it suffices to consider $N=1$.  Let $x\in (-1,1)$, then
	\begin{align*}
		K_\sigma(w(x))& = (x+1)^{2\sigma-1}\big(2\frac{x+1}{(x+1)^2}-1\big)^{\sigma-1}
		= (x+1)(1-x^2)^{\sigma-1},
	\end{align*}
	which is $\sigma$-harmonic in $(-1,1)$, by \cite[Corollary 1.7]{AJS17a}.  Then, by Proposition \ref{lem:sharm} and Lemma \ref{ptod}, $w$ is pointwisely $\sigma$-harmonic in $\{x>0\}$.
\end{proof}

We now show Proposition \ref{s-harmonic2}.

\begin{proof}[Proof of Proposition \ref{s-harmonic2}]
	By Lemmas \ref{lem:ldh} and \ref{s-harmonic1} it suffices to consider $N=1$ and $m\geq 1$. 
	Let $m\in\N$, $\sigma\in(0,1]$, $s=m+\sigma$, and let $u\in C^{\infty}((0,\infty))\cap \cL^1_{s}$ be such that 
	\begin{equation}\label{boundaries}
		\lim_{y\to\pm\infty}u^{(2i)}(y)\ \Big(\frac{d}{dy}\Big)^{2(m-i-1)+1}|y|^{-1-2\sigma} = 0,\quad 
		\lim_{y\to\pm\infty}u^{2i+1}(y)\ \Big(\frac{d}{dy}\Big)^{2(m-i-1)}|y|^{-1-2\sigma} =0
	\end{equation}
	for $i\in\{0,\ldots, m-1\}$. We now argue as in \cite[Theorem 1.2]{AJS17a}.  Note that, for $y\in\R^N\backslash\{0\}$,
	\begin{align}\label{art:1}
		|y|^{-1-2\sigma-2m}=\frac{-c_{1,\sigma}\pi^{1/2}\,\Gamma(-s)}{4^{s}\,\Gamma(\frac{1}{2}+s)}\Big(-\frac{d^2}{dy^2}\Big)^{m}|y|^{-1-2\sigma}\qquad \text{ for }y\neq 0.
	\end{align}
	To shorten notation, let 
	\begin{align*}
		\delta_mu(x,y)=\sum_{k=-{m-1}}^{m+1} (-1)^k { \binom{2m+2}{m+1-k}} u(x+ky)\quad \text{ and }\quad P:= \sum_{k=1}^{m+1} (-1)^k { \binom{2m+2}{m+1-k}}k^{2s}.
	\end{align*}
	By \cite[Lemma 2.2]{AJS17a} we know that
	\begin{align}\label{calc}
		\sum_{k=-m-1}^{m+1} (-1)^k { \binom{2m+2}{m+1-k}} k^{2m}=0.
	\end{align}
	Then, using \eqref{art:1}, integration by parts, \eqref{boundaries}, \eqref{calc}, the symmetry of the combinatorial coefficients, and changes of variables, we have, for $x>0$,
	\begin{align}
		(-\Delta)^su(x)&=\frac{c_{1,s}}{2}\int_{\R}\frac{\delta_mu(x,y)
		}{|y|^{1+2s}}\ dy=\frac{4^{s}\Gamma(\frac{1}{2}+s)}{ \pi^{\frac{1}{2}}\Gamma(-s) 2P}\int_{\R}\frac{
		\delta_mu(x,y)}{|y|^{1+2s}}\ dy\nonumber\\ 
	&=-\frac{c_{1,\sigma}}{2P}\lim_{r\to\infty}\int_{-r}^r(\delta_mu(x,y))\Big(-\frac{d^2}{dy^2}\Big)^{m}|y|^{-1-2\sigma}\ dy\\
	&=-\frac{c_{1,\sigma}}{P}\lim_{r\to\infty}\int_{-r}^r\Big(-\frac{d^2}{dy^2}\Big)^{m} \delta_mu(x,y) |y|^{-1-2\sigma}\ dy\nonumber\\
	&=-\frac{c_{1,\sigma}}{2P}\int_{\R} \sum_{k=-m-1}^{m+1} (-1)^k { \binom{2m+2}{m+1-k}}k^{2m} \frac{(-1)^{m}u^{(2m)}(x+ky)}{|y|^{1+2\sigma}}\ dy\nonumber\\
	&=\frac{(-1)^{m+1}c_{1,\sigma}}{2P}\int_{\R} \sum_{k=-m-1}^{m+1} (-1)^k { \binom{2m+2}{m+1-k}}k^{2m} \frac{u^{(2m)}(x+ky)-u^{(2m)}(x)}{|y|^{1+2\sigma}}\ dy\nonumber\\
	&=\frac{(-1)^{m+1}c_{1,\sigma}}{2P}\int_{\R} \sum_{k=-m-1}^{m+1} (-1)^k { \binom{2m+2}{m+1-k}}k^{2s} \frac{u^{(2m)}(x+y)-u^{(2m)}(x)}{|y|^{1+2\sigma}}\ dy\nonumber\\
	&=\frac{(-1)^{m+1}c_{1,\sigma}}{2P}\int_{\R} 2P \frac{u^{(2m)}(x+y)-u^{(2m)}(x)}{|y|^{1+2\sigma}}\ dy\nonumber\\
	&=c_{1,\sigma}\int_{\R} \frac{(-1)^{m}u^{(2m)}(x)-(-1)^{m}u^{(2m)}(x+y)}{|y|^{1+2\sigma}}\ dy = (-\Delta)^\sigma(-1)^{m}u^{(2m)}(x).\label{idlap}
\end{align}
Then, by Lemma \ref{s-harmonic1} and \eqref{idlap}, for $x>0$,
\begin{align}
	\label{sharm1d}\begin{aligned}
		(-\Delta)^s x^{2m+\sigma}_+=(-1)^m\prod_{j=1}^{2m}(j+\sigma)(-\Delta)^\sigma(x)^{\sigma}_+&=0,\\
		(-\Delta)^s x^{2m+\sigma-1}_+=(-1)^m\prod_{j=0}^{2m-1}(j+\sigma)(-\Delta)^\sigma(x)^{\sigma-1}_+&=0
	\end{aligned}
\end{align}
(note that \eqref{boundaries} is easily verified in these cases).  Let $j\in\{0,\ldots,2m-1\}$ and $u_j(x):=x_+^{j+\sigma-1}$.  Then by \eqref{sharm1d}, $u_j$ is $(\frac{j-1}{2}+\sigma)$-harmonic if $j$ is odd and 
$u_j$ is $(\frac{j}{2}+\sigma)$-harmonic if $j$ is even. By Lemma \ref{ptod} and because
\begin{align}\label{sm1tos}
	\int_{\R} u{(-\Delta)}^{s}\phi\ dx=\int_{\R} u{(-\Delta)}^{s-1}[-\phi^{(2)}]\ dx=0\qquad \text{ for any }\phi\in C^\infty_c((0,\infty)),
\end{align}
we have that $(s-1)$-harmonicity implies $s$-harmonicity; but then $u_j$ is $s$-harmonic for $j\in\{0,\ldots,2m+1\}$, as claimed.
\end{proof}

To close this subsection, we show that the $s$-harmonic functions in Proposition \ref{s-harmonic2} can be obtained using suitable boundary kernels. These kernels play a prominent role in the next subsection where we construct solutions of the nonhomogeneous problem \eqref{eq1}.

\begin{lemma}\label{lem:prop-edenhofer}
	Let $N\in\N$, $\sigma\in(0,1)$, $m\in \N$, and $j\geq m+1$. Then, for $x\in \Rp$,
	\begin{align*}
		\int_{\partial\Rp}\ \frac{(x_1)_+^{j+\sigma-1}}{|x-z|^{N+2(j-m-1)}}\ dz = \frac{\pi^{\frac{N-1}{2}}\Gamma(j-m-\frac{1}{2})}{\Gamma(\frac{N}{2}+j-m-1)}\ (x_1)_+^{2m-j+\sigma}.
	\end{align*}
\end{lemma}
\begin{proof}
	The statement for $N=1$ is clear, since, for $x>0$,
	\begin{align*}
		\int_{\partial \R_+}\ \frac{x^{j+\sigma-1}}{|x-z|^{1+2(j-m-1)}}\ dz=\frac{x^{j+\sigma-1}}{x^{1+2(j-m-1)}}= x^{2m-j+\sigma}.
	\end{align*}
	Now, assume that $N\geq 2$. The claim holds trivially if $x_1\leq 0$, so we may assume that $x=(x_1,x')\in \Rp$; then, by substitution with $z=(0,y+x')$ for $y\in \R^{N-1}$,
	\begin{align*}
		\int_{\partial\Rp}&\frac{x_1^{j+\sigma-1}}{|x-z|^{N+2(j-m-1)}}\ dz=\int_{\R^{N-1}} \frac{x_1^{2m-j+1+\sigma-N}}{(1+|y/x_1|^2)^{\frac{N}{2}+(j-m-1)}}\ dy=\int_{\R^{N-1}} \frac{x_1^{2m-j+\sigma}}{(1+|y|^2)^{\frac{N}{2}+(j-m-1)}}\ dy\\
		&= x_1^{2m-j+\sigma} \frac{2\pi^{\frac{N-1}{2}}}{\Gamma(\frac{N-1}{2})}\int_{0}^{\infty} \frac{\rho^{N-2}}{(1+\rho^2)^{\frac{N}{2}+(j-m-1)}}\ d\rho
		= x_1^{2m-j+\sigma} \frac{2\pi^{\frac{N-1}{2}}}{\Gamma(\frac{N-1}{2})}\frac{\Gamma(j-m-\frac{1}{2})\Gamma(\frac{N-1}{2})}{2\Gamma(\frac{N}{2}+j-m-1)}
	\end{align*}
	by equation \eqref{4674354}.
\end{proof}

\subsection{Boundary Poisson kernels}

We recall the definition of $E_{k,s}$ given in the introduction. For $m\in \N_0$, $\sigma\in(0,1]$, $s=m+\sigma$, $k\in\{0,\ldots,m\}$, $y\in \partial\Rp$, and $x\in \R^N\backslash\{y\}$, let
\begin{align*}
	E_{k,s}(x,y)&:=\sum_{i = 0}^{\lfloor\frac{m-k}{2}\rfloor}\alpha_{m-k,i}\frac{(x_1)_+^{s+m-k-2i}}{|y-x|^{N+2(m-k-i)}},
	\qquad \alpha_{l,i}:=\frac{(-1)^i2^{l-2i}\Gamma(\frac{N}{2}+l-i)}{\pi^{\frac{N}{2}}(l-2i)!\ i!}
\end{align*}
and $\lfloor a \rfloor$ is the integer part of $a$. The main objective of this section is to prove Theorem \ref{exp:sol:thm}.  This requires integral and combinatorial identities, regularity estimates, and to show $s$-harmonicity via the $K_s$ transform.  We split these steps into several lemmas. 

\begin{lemma}\label{int:lem} Let $N\geq 2$, $m\in\N$, $k\in\{1,\ldots,m-1\}$, $j\in \{k+1,\ldots,m\}$ such that $j-k$ is even, $i\in\{0,\ldots,
	\frac{m-k}{2}\}$, and $\gamma\in\N^{N-1}$ such that $|\gamma|=\frac{j-k}{2}$, then
	\begin{align}\label{formula}
		\int_{\R^{N-1}}\frac{y^{2\gamma}}{(1+|y|^2)^{\frac{N}{2}+(m-k-i)}}\ dy
		=\frac{\pi^{\frac{N-1}{2}}(2\gamma)!}{2^{j-k}\gamma!}
		\frac{\Gamma \left(m-i+\frac{1-j-k}{2}\right)}{\Gamma \left(m-i-k+\frac{N}{2}\right)}.
	\end{align}
\end{lemma}
\begin{proof}
	If $N=2$ then $\gamma=\frac{j-k}{2}$, and therefore
	\begin{align*}
		\int_{\R^{N-1}}\frac{y^{2\gamma}}{(1+|y|^2)^{\frac{N}{2}+(m-k-i)}}\ dy
		&=2\int_{0}^\infty\frac{y^{j-k}}{(1+|y|^2)^{\frac{N}{2}+(m-k-i)}}\ dy\\
		&=\Gamma \left(\frac{j-k+1}{2}\right)\frac{\Gamma\left(m-i+\frac{1-j-k}{2}\right)}{\Gamma (m-i-k+1)},
	\end{align*}
	and \eqref{formula} follows from the identity
	\begin{align}\label{gamma:id}
		\Gamma\bigg(\frac{1}{2}+a\bigg)=2^{1-2a}\pi^{\frac{1}{2}}\frac{\Gamma(2a)}{\Gamma(a)}\qquad \text{ for }a>0,
	\end{align}
	since
	\begin{align*}
		\Gamma \left(\frac{j-k+1}{2}\right)=
		\Gamma \left(\gamma+\frac{1}{2}\right)
		=\frac{\pi^{\frac{1}{2}}(2\gamma)!}{2^{j-k}\gamma!}.
	\end{align*}
	
	If $N=3$, then $\gamma_1+\gamma_2=\frac{j-k}{2}$ for some $\gamma_1,\gamma_2\in\N_0$, and, using polar coordinates,
	\begin{align*}
		\int_{\R^{N-1}}\frac{y^{2\gamma}}{(1+|y|^2)^{\frac{3}{2}+(m-k-i)}}\ dy
		&=\int_0^\infty \frac{r^{1+j-k}}{(1+r^2)^{\frac{3}{2}+(m-k-i)}}\ dr \int_{0}^{2\pi}\sin^{2\gamma_1}\theta\cos^{2\gamma_2}\theta\ d\theta\\
		&=\Gamma \left(\frac{j-k+2}{2}\right) \frac{\Gamma
			\left(m-i+\frac{1-j-k}{2}\right)}{2 \Gamma \left(m-i-k+\frac{3}{2}\right)}
		\frac{2 \Gamma
			\left(\gamma_1+\frac{1}{2}\right) \Gamma \left(\gamma_2+\frac{1}{2}\right)}{\Gamma
			(\gamma_1+\gamma_2+1)},
	\end{align*}
	and \eqref{formula} follows since, by \eqref{gamma:id},
	\begin{align*}
		\Gamma\left(\gamma_1+\frac{1}{2}\right) \Gamma \left(\gamma_2+\frac{1}{2}\right)
		=\frac{\pi(2\gamma_1)!(2\gamma_2)!}{2^{j-k}\gamma_1!\gamma_2!} =\frac{\pi(2\gamma)!}{2^{j-k}\gamma!}.
	\end{align*}
	Finally, for $N\geq 4$ we can argue similarly using using spherical coordinates in $\R^{N-1}$ (for more details see the proof of \cite[Lemma 3.3]{AJS17b}, for example), that is, $y_{i}=r\cos \theta_{i} \prod_{l=1}^{i-1}\sin\theta_l$ for $i\in\{1,\ldots,N-2\}$ and $y_{N-1}=r \prod_{l=1}^{N-2}\sin\theta_l$, where 
	$r>0,$ $\theta_1,\ldots,\theta_{N-3}\in(0,\pi),$ $\theta_{N-2}\in(0,2\pi)$, and the associated Jacobian is 
	$J(r,\theta_1,\ldots,\theta_{N-2})=r^{N-2}\prod_{j=1}^{N-3}\sin^{N-2-j}\theta_j.$ Then
	\begin{align*}
		\int_{\partial\Rp}\frac{y^{2\gamma}}{|1+|y'|^2|^{\frac{N}{2}+(m-k-i)}}\ dy
		&=\frac{\pi^{\frac{N-1}{2}}(2\gamma)!}{2^{j-k-1}\gamma!\Gamma(\frac{N-1+j-k}{2})}
		\int_{0}^\infty \frac{r^{N-2+j-k}}{|1+r^2|^{\frac{N}{2}+(m-k-i)}}\ dr\\
		&=\frac{\pi^{\frac{N-1}{2}}(2\gamma)!}{2^{j-k-1}\gamma!\Gamma(\frac{N-1+j-k}{2})}
		\frac{\Gamma \left(\frac{1}{2} (j-k+N-1)\right) \Gamma \left(m-i+\frac{1-j-k}{2}\right)}{2 \Gamma \left(m-i-k+\frac{N}{2}\right)}\\
		&=\frac{\pi^{\frac{N-1}{2}}(2\gamma)!}{2^{j-k}\gamma!}
		\frac{\Gamma \left(m-i+\frac{1-j-k}{2}\right)}{\Gamma \left(m-i-k+\frac{N}{2}\right)}.
	\end{align*}
\end{proof}

\begin{lemma}\label{last:lem}
	Let $m\in\N$, $k\in\{1,\ldots,m-1\}$, and $j\in \{k,\ldots,m\}$ such that $j-k$ is even, then
	\begin{align*}
		\sum_{i = 0}^{\lfloor\frac{m-k}{2}\rfloor}
		\alpha_{m-k,i}
		\frac{\Gamma \left(m-i+\frac{1-j-k}{2}\right)}{\Gamma \left(m-i-k+\frac{N}{2}\right)}=\delta_{j,k}\pi^{\frac{1-N}{2}}.
	\end{align*}
\end{lemma}
\begin{proof} 
	We use the following combinatorial identity (see, for example, \cite[equation (3.63)]{G72})
	\begin{align}\label{Gould}
		\sum_{i=0}^{\lfloor \frac{l}{2} \rfloor}
		(-1)^i \binom{x}{i}\binom{2x-2i}{l-2i}=\binom{x}{l}2^l
		\qquad \text{ for }\quad l,x\in\N,
	\end{align}
	where $\binom{x}{l}:=0$ if $x<l$. Let $l=m-k$ and, since $j-k$ is even, let $j-k=2a$ for some $a\in\N_0$, then, using \eqref{gamma:id} and \eqref{Gould},
	\begin{align*}
		\sum_{i = 0}^{\lfloor\frac{l}{2}\rfloor}
		\alpha_{l,i}&
		\frac{\Gamma \left(l-i+\frac{1-j+k}{2}\right)}{\Gamma \left(l-i+\frac{N}{2}\right)}
		=\sum_{i = 0}^{\lfloor\frac{l}{2}\rfloor}
		\frac{(-1)^i2^{l-2i}}{\pi^{\frac{N}{2}}(l-2i)!\ i!}
		\Gamma \left(l+\frac{1}{2}-i-a\right)\\
		&=\frac{2^{2a-l}}{\pi^{\frac{N-1}{2}}}
		\sum_{i = 0}^{\lfloor\frac{l}{2}\rfloor}
		\frac{(-1)^i}{(l-2i)!\ i!}
		\frac{(2(l-i-a))!}{(l-i-a)!}
		=\frac{2^{2a-l}}{\pi^{\frac{N-1}{2}}}\sum_{i = 0}^{\lfloor\frac{l}{2}\rfloor} (-1)^i
		\binom{2(l-i-a)}{l-2i}\frac{(l-2a)!}{(l-i-a)!\ i!} \\
		&=\frac{2^{2a-l}}{\pi^{\frac{N-1}{2}}} \frac{(l-2a)!}{(l-a)!}\sum_{i = 0}^{\lfloor\frac{l}{2}\rfloor} (-1)^i
		\binom{2(l-a)-2i}{l-2i}\binom{l-a}{i}
		=\frac{2^{2a-l}}{\pi^{\frac{N-1}{2}}} \frac{(l-2a)!}{(l-a)!}\binom{l-a}{l}2^l=\delta_{j,k}\pi^{\frac{1-N}{2}}.
	\end{align*}
\end{proof}

\begin{lemma}\label{Holder:lem:1}
	Let $N\geq 2$, $l,k,j,m,i\in\N_0$ such that $0\leq l\leq  k\leq j\leq m$, $i\leq \lfloor\frac{m-k}{2}\rfloor$, $g\in C_c^\beta(\partial \Rp)$ for some $0<\beta<1$,
	$R>0$ such that $\supp(g)\subset B_R(0)\subset \R^{N-1}$, and 
	\begin{align}\label{L:def}
		L:=\{x\in\R^N\::\:0\leq x_1\leq 1,\ { |x'|<R+1}\}.
	\end{align}
	For $x\in \Rp$ and $\alpha\in\N_0^{N}$ with $|\alpha|=j-l$ let
	\begin{align*}
		w(x):=x_1^{k-l}\int_{\partial\Rp}\frac{g(x_1y+x')y^\alpha}{|1+|y|^2|^{\frac{N}{2}+(m-k-i)}}\ dy.
	\end{align*} 
	Then there is $C(g,N)=C>0$ such that
	\begin{align}\label{hm2}
		\sup_{{x,\widetilde x\in L},\ x_1\neq \widetilde x_1}\frac{|w(x)-w(\widetilde x)|}{|x-\widetilde x|^\beta}<C.
	\end{align}
\end{lemma}
\begin{proof}
	Since $g$ has compact support it is clear that $w(x)<\infty$ for all $x\in\Rp$.  In the following $C>0$ denotes possibly different positive constants depending only on $g$ and $N$. Let
	\begin{align*}
		b:=1+2m-j-k-2i=1+(m-j)+(m-k-2i)\geq 1,
	\end{align*}
	and $R>0$ and $L$ as in the statement. Then, for ${x,\widetilde x\in L}$, $x\neq\widetilde x$, $x_1\leq \widetilde x_1$,
	\begin{align}
		& \frac{|w(x_1,x')-w(\widetilde x_1,x')|}{|x_1-\widetilde x_1|^\beta}
		\leq \int_{B_{\frac{R}{x_1}(x')}}\frac{|y|^{j-l}}{|1+|y|^2|^{\frac{N}{2}+(m-k-i)}}
		\frac{|x_1^{k-l}g(x_1y+x')-\widetilde x_1^{k-l}g(\widetilde x_1y + x')|}{|x_1-\widetilde x_1|^\beta}
		\ dy\nonumber\\
		&\qquad \leq \int_{B_{\frac{R}{x_1}(x')}}\frac{|y|^{j-l}}{|1+|y|^2|^{\frac{N}{2}+(m-k-i)}}
		\frac{|(x_1^{k-l}-\widetilde x_1^{k-l})g(\widetilde x_1y+x')+ x_1^{k-l}(g(x_1y+x')-g(\widetilde x_1y + x'))|}{|x_1-\widetilde x_1|^\beta}
		\ dy\nonumber\\
		&\qquad \leq C\int_{B_{\frac{R}{x_1}(x')}}\frac{|y|^{j-l}}{|1+|y|^2|^{\frac{N}{2}+(m-k-i)}}
		\Big(\frac{|x_1^{k-l}-\widetilde x_1^{k-l}|}{|x_1-\widetilde x_1|^{\beta}}+ x_1^{k-l}|y|^\beta\Big)\ dy.\label{h2}
	\end{align}
	Note that
	
	\begin{align}
		\sup_{\substack{{x,\widetilde x\in L}\\ x_1\leq \widetilde x_1}}\ \ &\int_{B_{\frac{R}{x_1}(x')}}
		\frac{|y|^{j-l}}{|1+|y|^2|^{\frac{N}{2}+(m-k-i)}}
		x_1^{k-l}|y|^\beta\ dy
		\leq C 
		\sup_{\substack{{x,\widetilde x\in L}\\ x_1\leq \widetilde x_1}}
		x_1^{k-l}
		\int_{0}^{\frac{R}{x_1}+|x'|}\frac{r^{\beta +j-l+N-2}}{|1+r^2|^{\frac{N}{2}+(m-k-i)}}\ dr \\
		&\leq C
		\sup_{\substack{{x,\widetilde x\in L}\\ x_1\leq \widetilde x_1}}
		x_1^{k-l} \Big(1+
		\int_{1}^{\frac{R}{x_1}+|x'|}r^{ \beta +(k-l)-1-(1+2m-j-k-2i)}\ dr\Big) \\
		&\leq C
		\sup_{\substack{{x,\widetilde x\in L}\\ x_1\leq \widetilde x_1}}
		x_1^{k-l}+
		(R+|x'|x_1)^{k-l}\int_{1}^{\infty}r^{\beta-b-1}\ dr \leq C
		\sup_{\substack{{x,\widetilde x\in L}\\ x_1\leq \widetilde x_1}}
		x_1^{k-l}+
		\frac{(R+|x'|x_1)^{k-l}}{b-\beta}<\infty\label{h3}
	\end{align}

	and, similarly, assuming without loss of generality that $k-l\geq 1$,
	\begin{align}
		\sup_{\substack{{x,\widetilde x\in L}\\ x_1\leq \widetilde x_1}}
		& \quad\int_{B_{\frac{R}{x_1}(x')}}\frac{|y|^{j-l}}{|1+|y'|^2|^{\frac{N}{2}+(m-k-i)}}
		\frac{|x_1^{k-l}-\widetilde x_1^{k-l}|}{|x_1-\widetilde x_1|^{\beta}}\ dy 
		\ \leq \\
		& \leq C\ \sup_{\substack{{x,\widetilde x\in L}\\ x_1\leq\widetilde x_1}}\ 
		\frac{|x_1^{k-l}-\widetilde x_1^{k-l}|}{|x_1-\widetilde x_1|^{\beta}}
		(1+\int_{1}^{\frac{R}{x_1}+|x'|}r^{(k-l-\beta)-b-1+\beta}\ dr) \\
		& \leq C\ \sup_{\substack{{x,\widetilde x\in L}\\ x_1\leq\widetilde x_1}}\ 
		\frac{|x_1^{k-l}-\widetilde x_1^{k-l}|}{|x_1-\widetilde x_1|^{\beta}}
		\Bigg(1+\frac{x_1^{l-k+\beta}(R+|x'|x_1)^{k-l-\beta}}{b-\beta}\Bigg) \\
		& \leq C\ \sup_{\substack{{x,\widetilde x\in L}\\ x_1\leq\widetilde x_1}}\ 
		\frac{|1-(\frac{\widetilde x_1}{x_1})^{k-l}|}{|1-(\frac{\widetilde x_1}{x_1})|^{\beta}}(x_1^{k-l -\beta}+C)<\infty\label{h4}.
	\end{align}
	
	Finally, arguing as in \eqref{h3},
	\begin{align}
		\sup_{\substack{{x,\widetilde x\in L}\\ x_1\leq\widetilde x_1}}\frac{|w(\tilde{x_1},x')-w(\tilde{x_1},\tilde{x}\ ')|}{|x'-\tilde{x}\ '|^{\beta}}
		&\leq\sup_{\substack{{x,\widetilde x\in L}\\ x_1\leq\widetilde x_1}}
		x_1^{k-l}\int_{B_{\frac{R}{x_1}(x')}}\frac{|y|^{j-l}}{|1+|y'|^2|^{\frac{N}{2}+(m-k-i)}}
		\frac{|g(\widetilde x_1y+x')-g(\widetilde x_1 y + \widetilde x')|}{|x_1-\widetilde x_1|^\beta}
		\ dy\nonumber\\
		&\leq C\sup_{\substack{{x,\widetilde x\in L}\\ x_1\leq\widetilde x_1}} 
		x_1^{k-l}\int_{B_{\frac{R}{x_1}(x')}}\frac{|y|^{j-l}}{|1+|y'|^2|^{\frac{N}{2}+(m-k-i)}}
		\ dy<\infty.\label{h5}
	\end{align}
	Since 
	\begin{align*}
		\frac{|w(x_1,x')-w(\tilde{x_1},\tilde{x}\ ')|}{|x-\tilde{x}|^{\beta}}\leq \frac{|w(x_1,x')-w(\tilde{x_1},x')|}{|x_1-\tilde{x_1}|^{\beta}}+\frac{|w(\tilde{x_1},x')-w(\tilde{x_1},\tilde{x}\ ')|}{|x'-\tilde{x}\ '|^{\beta}}
	\end{align*}
	the claim \eqref{hm2} now follows from \eqref{h2},\eqref{h3},\eqref{h4},\eqref{h5}.
\end{proof}

\begin{lemma}\label{Holder:lem}
	Let $N\geq 2$, $l,k,j,m,i\in\N_0$ such that $0\leq l\leq  k\leq j\leq m$, $i\leq \lfloor\frac{m-k}{2}\rfloor$, $h\in C_c^{j+\beta}(\partial \Rp)$ for some $0<\beta<1$, let $w:\R^N\to \R$ be given by
	\begin{align*}
		w(x)=\partial_{x_1}^{j}\Bigg[\int_{\partial \Rp} \frac{(x_1)_+^{k}h(x_1y+x')}{(1+|y|^2)^{\frac{N}{2}+m-k-i}} \;dy\Bigg]\quad\text{ for $x\in \R^N$}
	\end{align*}
	Then there is $C(N,h)=C>0$ such that
	\begin{align}\label{hm1}
		\sup_{{x,\widetilde x\in M},\ x\neq \widetilde x}\frac{|w(x)-w(\widetilde x)|}{|x-\widetilde x|^\beta}<C,\qquad M:=\{x=(x_1,x')\in\R^N\::\: 0<x_1< 1\},
	\end{align}
	and
	\begin{align}\label{hm11}
		\sup_{x_1,\widetilde x_1\in(0,1)}\frac{|w(x_1,x')-w(\widetilde x_1,x')|}{|x_1-\widetilde x_1|^\beta}<\frac{C}{1+|x'|^N}\qquad \text{ for all }x'\in \R^{N-1}.
	\end{align}
\end{lemma}
\begin{proof} 
	Observe that, for $x\in \Rp$,
	\begin{align*}
		w(x)&=\sum_{l=0}^k\frac{k!}{(k-l)!}x_1^{k-l}\int_{\partial\Rp}\frac{\partial_{x_1}^{j-l}[h(x_1y+x')]}{|1+|y'|^2|^{\frac{N}{2}+(m-k-i)}}\ dy\\
		& =\sum_{l=0}^k\frac{k!}{(k-l)!}x_1^{k-l}\int_{\partial\Rp}\frac{
			\sum_{|\alpha|=j-l}\partial^\alpha h(x_1y+x')y^\alpha
		}{|1+|y'|^2|^{\frac{N}{2}+(m-k-i)}}\ dy.
	\end{align*}
	Let $R>0$ be such that $\supp(h)\subset B_R(0)\subset \R^{N-1}$ and $L$ as in \eqref{L:def}. Then, by Lemma \ref{Holder:lem:1}, we have that
	\begin{align}\label{hm2p}
		\sup_{{x,\widetilde x\in L},\ x_1\neq \widetilde x_1}\frac{|w(x)-w(\widetilde x)|}{|x-\widetilde x|^\beta}<C.
	\end{align}
	Let $x\in L_0:=\{x\in \R^N\;:\; 0\leq x_1\leq 1,\ |x'|\geq R+1\}$, then using the substitution $z=x_1y+x'$ we have
	\begin{align*}
		\partial_{x_1}^{j}\Bigg[\int_{\partial \Rp} \frac{(x_1)_+^{k}h(x_1y+x')}{(1+|y|^2)^{\frac{N}{2}+m-k-i}} \;dy\Bigg]
		=\int_{B_R(0)} \partial_{x_1}^{j}\Bigg[\frac{(x_1)_+^{1+2m-k-2i}}{|x-z|^{N+2(m-k-i)}} \Bigg]h(z)\;dz
	\end{align*}
	Fix $z\in B_R(0)$ and note that, for $x\in L_0$, we have $|x-z|\geq 1$ and therefore 
	\begin{align*}
		\text{$x\mapsto \eta(x):=\partial_{x_1}^{j}\Bigg[\frac{(x_1)_+^{1+2m-k-2i}}{|x-z|^{N+2(m-k-i)}} \Bigg]$ belongs to $C^\infty(L_0)$}\quad \text{ and }\quad
		|\nabla\eta(x)|\leq \frac{K}{|x-z|^N} \text{ for some }K>0.
	\end{align*}
	But then, for $x,\tilde{x}\in L_0$, $x\neq \tilde{x}$,
	\begin{align}
		\frac{|w(x)-w(\tilde{x})|}{|x-\tilde{x}|^{\beta}}&\leq K|x-\tilde{x}|^{1-\beta}\int_{\partial \Rp}\frac{|h(z)|}{|z-x|^{N}}\ dz\leq K\|h\|_{\infty} \int_{B_R(0)}|x-z|^{-N}\ dz\nonumber\\
		&\leq K\|h\|_{\infty}|B_R(0)|\operatorname{dist}(x,\partial B_R(0))^{-N},\label{hm3}
	\end{align}
	and \eqref{hm11} follows from \eqref{hm2p} and \eqref{hm3}.  Finally, let $M:=\{0<x_1<1\}$, then
	\begin{align*}
		\sup_{{x,\widetilde x\in M},\ x\neq \widetilde x}\frac{|w(x)-w(\widetilde x)|}{|x-\widetilde x|^\beta}
		\leq \sup_{{x,\widetilde x\in M},\ x\neq \widetilde x}\frac{|w(x)-w(0,x')|+|w(0,\widetilde x')-w(\widetilde x)|}{|x-\widetilde x|^\beta}+\frac{|h(x')-h(\widetilde x)|}{|x-\widetilde x|^\beta}<C
	\end{align*}
	by \eqref{hm2p},\eqref{hm3}, and the proof is finished.
\end{proof}

\begin{lemma}\label{s-harmonic-gen1-compact}
	Let $N\in \N$, $m\in \N_0$, $\sigma\in(0,1]$, $s=m+\sigma$, and $j\in\{0,\ldots,2m+1\}$. If $h\in L^{\infty}(\partial \Rp)$ with compact support, then the function $u:\R^N\to \R$ given by
	\begin{align*}
		u(x)=\int_{\partial \Rp} \frac{(x_1)_+^{j+\sigma-1}}{|x-z|^{N+2(j-m-1)}}h(z) \;dz\quad\text{ for $x\in \R^N$}
	\end{align*}
	is in $\cL^1_s \cap C^{\infty}(\Rp)$ and $u$ is pointwisely $s$-harmonic in $\Rp$. In particular,
	\begin{align*}
		v(x)=(x_1)_+^s\int_{\partial \Rp} \frac{(x_1)_+^{k}}{|x-z|^{N+2k}}h(z) \;dz\quad\text{ for $x\in \R^N$}
	\end{align*}
	is pointwisely $s$-harmonic in $\Rp$ for any $k\in\{-m-1,\ldots,0,\ldots,m\}$.
\end{lemma}
\begin{proof}
	Observe that
	\begin{align*}
		\int_{\R^N}\frac{(x)_+^{2m-j+\sigma}}{1+|x|^{1+2s}}\ dx
		=\int_{0}^\infty\frac{x^{2m-j+\sigma}}{1+x^{1+2s}}\ dx
		\leq \int_{0}^1\frac{x^{2m-j+\sigma}}{1+x^{1+2s}}\ dx
		+\int_{1}^\infty x^{-j-\sigma-1}\ dx<\infty.
	\end{align*}
	Then $(x)_+^{2m-j+\sigma}\in \cL^1_s(\R)$ and, by Lemma \ref{lem:ldh}, 
	\begin{align}\label{cref}
		(x_1)_+^{2m-j+\sigma}\in \cL^1_s(\R^N) .
	\end{align}
	Moreover, by Lemma \ref{lem:prop-edenhofer},
	\begin{align*}
		|u(x)|\leq \|h\|_{L^\infty(\partial\Rp)} \int_{\partial \Rp} \frac{(x_1)_+^{j+\sigma-1}}{|x-z|^{N+2(j-m-1)}} \;dz\leq C (x_1)_+^{2m-j+\sigma},
	\end{align*}
	for some $C(h,N,s)=C>0$. Therefore, using \eqref{cref},
	\begin{align*}
		\int_{\R^N}\frac{|u(x)|}{1+|x|^{N+2s}}\leq C\int_{\R^N}\frac{(x_1)_+^{2m-j+\sigma}}{1+|x|^{N+2s}}<\infty
	\end{align*}
	and thus $u\in\cL^1_s$.  Let $K:=\operatorname{supp}(h)\subset \partial \Rp$. To show that $u$ is $s$-harmonic, let $j\in\{0,\ldots,2m+1\}$ and $\kappa$, $K_s$ as in \eqref{kappa:K:def} with $v=0$ and $c=1$.  Then 
	\begin{align*}
		\kappa x=\frac{x}{|x|^2}\qquad \text{ and }\qquad K_s((x_1)_+^{j+\sigma-1})=\frac{(x_1)_+^{j+\sigma-1}}{|x|^{N+2(j-m-1)}}
	\end{align*}
	is distributionally $s$-harmonic in $\kappa(\Rp)=\Rp$, by Propositions \ref{lem:sharm} and \ref{s-harmonic2}. Therefore, by the Fubini's theorem, for any $\psi\in C^\infty_c(\Rp)$
	\begin{align*}
		&\int_{\Rp}u(x)(-\Delta)^{s}\psi(x)\ dx=\int_{\Rp}\int_{\partial \Rp}\frac{(x_1)_+^{j+\sigma-1}}{|x-z|^{N+2(j-m-1)}}h(z)\ dz (-\Delta)^s\psi(x)\ dx\\
		&=\int_{\partial \Rp} \int_{\Rp} \frac{(x_1)_+^{j+\sigma-1}}{|x|^{N+2(j-m-1)}} (-\Delta)^s\psi(x+z)\ dx\ h(z) \ dz=0,
	\end{align*}
	which implies that $u$ is distributionally $s$-harmonic in $\Rp$. If $N=1$, then $u\in C^{\infty}(\Rp)$,  by definition, and if $N\geq 2$ then 
	\begin{align*}
		\partial^{\alpha} u(x)= \int_{K} \partial^{\alpha}_x\Big(\frac{(x_1)_+^{j+\sigma-1}}{|x-z|^{N+2(j-m-1)}}\Big)h(z)\ dz\qquad \text{for $x\in\Rp$ and $\alpha\in \N_0^N$},
	\end{align*}
	and then $u\in C^{\infty}(\Rp)$ also in this case.  Therefore, by Lemma \ref{ptod}, $u$ is pointwisely $s$-harmonic. 
\end{proof}

We are ready to show the main theorem of this section.

\begin{thm}\label{thm:edenhofer}
	Let $m\in\N_0$, $\sigma,\beta\in(0,1]$, $s=m+\sigma$, $k\in \{0,\ldots,m\}$, $h\in C_c^{m+\beta}(\partial\Rp)$, and let $v_k:\R^N\to \R$ be given by
	\begin{equation*}
		v_k(x)=\int_{\partial\Rp}E_{k,s}(x,y) h(y)\ dy\qquad \text{ for }x\in\R^N.
	\end{equation*} 
	Then $v_k\in \cL^1_s \cap C^{\infty}(\Rp)$ is pointwisely $s$-harmonic in $\Rp$ and
	\begin{equation}\label{lem:edenhofer-claim}
		D^{j+\sigma-1}v_k=\delta_{j,k}h\quad \text{ on $\partial\Rp$ \ \  for $j\in\{0,\ldots,m\}$.}
	\end{equation}
	Moreover, there is $C=C(N,h,k)>0$ such that, for all $x'\in\R^{N-1}$,,
	\begin{align}\label{reg:claim} 
		\sup_{x_1,\widetilde x_1\in(0,1),\ x_1\neq \widetilde x_1}
		\frac{|\partial_{x_1}^m[(x_1)_+^{1-\sigma} v_k(x_1,x')]-\partial_{x_1}^m[(\widetilde x_1)_+^{1-\sigma} v_k(\widetilde x_1,x')]|}{|x_1-\widetilde x_1|^\beta}\leq \frac{C}{1+|x'|^N}
	\end{align}
	and, for $k\in \{0,\ldots,m\}$,
	\begin{equation}\label{lem:edenhofer-decay}
		|v_k(x)|\leq C\frac{(x_1)_+^{\sigma+k-1}}{1+|x|^{N+m-k}} \qquad \text{ for $x\in \R^N$.}
	\end{equation}
\end{thm}
\begin{proof}
	If $N=1$ and $k\in \{0,\ldots,m\}$, then, by Lemma \ref{last:lem} (with $j=k$), we have for $x>0$ that
	\begin{align*}
		v_k(x)&=\sum_{i = 0}^{\lfloor\frac{m-k}{2}\rfloor}\alpha_{m-k,i}  \frac{x^{s+m-k-2i}}{x^{1+2(m-k-i)}} h(0)=x^{s-m+k-1}h(0)\sum_{i = 0}^{\lfloor\frac{m-k}{2}\rfloor}\alpha_{m-k,i}=x^{s-m+k-1}h(0),
	\end{align*}	
	which is $s$-harmonic, by Proposition \ref{s-harmonic2}, and \eqref{lem:edenhofer-claim}, \eqref{reg:claim}, \eqref{lem:edenhofer-decay} are clearly satisfied. Now, let $N\geq 2$, $k\in \{0,\ldots,m\}$, and note that
	\begin{align*}
		v_k(x)=\sum_{i = 0}^{\lfloor\frac{m-k}{2}\rfloor}\alpha_{m-k,i}x_1^{s-i}\int_{\partial\Rp}\frac{(x_1)_+^{m-k-i}}{|y-x|^{N+2(m-k-i)}} h(y)\ dy.
	\end{align*}
	Since $m-k-i < m+\sigma-i=s-i$, Lemma~\ref{s-harmonic-gen1-compact} implies that each summand is $(s-i)$-harmonic, which in turn implies that $v_k$ is pointwisely $s$-harmonic in $\Rp$ (as we argued in \eqref{sm1tos}). 
	
	We now show \eqref{lem:edenhofer-decay}. Observe that, by a change of variables,
	\begin{align*}
		|v_k(x)|\leq x_1^{\sigma+k-1}\sum_{i = 0}^{\lfloor\frac{m-k}{2}\rfloor}\alpha_{m-k,i}\int_{\partial\Rp}\frac{|h(x_1y+x')|}{|1+|y'|^2|^{\frac{N}{2}+(m-k-i)}}\ dy\leq Cx_1^{\sigma+k-1}.
	\end{align*}
	But then, \eqref{lem:edenhofer-decay} follows from $(1+|y'|^2)^{\frac{N}{2}+m-k-i}\geq (1+|y'|^2)^{\frac{N+m-k}{2}}$
	and the fact that $h$ is uniformly bounded. The claim \eqref{reg:claim} follows from Lemma \ref{Holder:lem} and the fact that, for $x\in\Rp$,
	\begin{align}\label{ref1}
		x_1^{1-\sigma}v_k(x)= \sum_{i = 0}^{\lfloor\frac{m-k}{2}\rfloor}\alpha_{m-k,i}\int_{\partial\Rp}\frac{x_1^{k}h(x_1y+x')}{|1+|y'|^2|^{\frac{N}{2}+(m-k-i)}}\ dy
	\end{align}
	and 
	\begin{align*}
		\partial_{1}^m[x_1^{1-\sigma}v_k(x)]
		=\sum_{i = 0}^{\lfloor\frac{m-k}{2}\rfloor}\frac{\alpha_{m-k,i}}{m!}\sum_{l=0}^k\frac{k!}{(k-l)!}x_1^{k-l}\int_{\partial\Rp}\frac{\partial_{x_1}^{m-l}[h(x_1y+x')]}{|1+|y'|^2|^{\frac{N}{2}+(m-k-i)}}\ dy.
	\end{align*}
	
	Finally, we show \eqref{lem:edenhofer-claim}. Fix $x'\in\partial\Rp$ and let $j\in\{k,\ldots,m\}$, then, by \eqref{ref1},
	\begin{align}
		D^{j+\sigma-1}v_k(x')&=\frac{1}{k!}{\lim_{x_1\to 0^+}}\partial_1^{k}[x_1^{1-\sigma}v_k(x')]\nonumber\\
		&=\lim_{x_1\to 0}\sum_{i = 0}^{\lfloor\frac{m-k}{2}\rfloor}\frac{\alpha_{m-k,i}}{j!}\int_{\partial\Rp}\frac{\partial_{x_1}^j(x_1^{k}h(x_1y+x'))}{|1+|y'|^2|^{\frac{N}{2}+(m-k-i)}}\ dy\label{ref3}\\
		&=\lim_{x_1\to 0}\sum_{i = 0}^{\lfloor\frac{m-k}{2}\rfloor}\frac{\alpha_{m-k,i}}{j!}\sum_{l=0}^k\frac{k!}{(k-l)!}x_1^{k-l}\int_{\partial\Rp}\frac{\partial_{x_1}^{j-l}[h(x_1y+x')]}{|1+|y'|^2|^{\frac{N}{2}+(m-k-i)}}\ dy\nonumber\\
		&=\sum_{|{\gamma}|=j-k}\partial^{\gamma} h(x')\sum_{i = 0}^{\lfloor\frac{m-k}{2}\rfloor}\frac{\alpha_{m-k,i}k!}{j!}\int_{\partial\Rp}\frac{y^{\gamma}}{|1+|y'|^2|^{\frac{N}{2}+(m-k-i)}}\ dy.\label{ref2}
	\end{align}
	Observe that the integral in \eqref{ref2} is finite, by Lemma~\ref{int:lem}, and the interchange between derivative and integral in \eqref{ref3} can be justified as in the proof of Lemma~\ref{Holder:lem} using that $h$ is compactly supported. 
	
	If there is some ${\gamma}_i$ odd, then by a change of variables,
	\begin{align*}
		\int_{\R^{N-1}}\frac{y^{\gamma}}{|1+|y|^2|^{\frac{N}{2}+(m-k-i)}}\ dy=-\int_{\R^{N-1}}\frac{y^{\gamma}}{|1+|y|^2|^{\frac{N}{2}+(m-k-i)}}\ dy, 
	\end{align*}
	that is, $\int_{\R^{N-1}}\frac{y^\beta}{|1+|y|^2|^{\frac{N}{2}+(m-k-i)}}\ dy=0$.  So we may assume that $j-k$ is even.  Then, by Lemmas~\ref{int:lem} and~\ref{last:lem}, 
	\begin{align*}
		D^{j+\sigma-1}v_k(x')&=\sum_{|\gamma|=\frac{j-k}{2}}\partial^{2\gamma} h(x')\sum_{i = 0}^{\lfloor\frac{m-k}{2}\rfloor}\frac{\alpha_{m-k,i}k!}{j!}\int_{\partial\Rp}\frac{y^{2\gamma}}{|1+|y'|^2|^{\frac{N}{2}+(m-k-i)}}\ dy\\
		&=\sum_{|\gamma|=\frac{j-k}{2}}\partial^{2\gamma} h(x')\sum_{i = 0}^{\lfloor\frac{m-k}{2}\rfloor}\frac{\alpha_{m-k,i}k!}{j!}\frac{\pi^{\frac{N-1}{2}}(2\gamma)!}{2^{j-k}\gamma!}
		\frac{\Gamma \left(m-i+\frac{1-j-k}{2}\right)}{\Gamma \left(m-i-k+\frac{N}{2}\right)}
		\\
		&=\sum_{|\gamma|=\frac{j-k}{2}}\partial^{2\gamma} h(x')\frac{k!}{j!}\frac{\pi^{\frac{N-1}{2}}(2\gamma)!}{2^{j-k}\gamma!}
		\sum_{i = 0}^{\lfloor\frac{m-k}{2}\rfloor}
		\alpha_{m-k,i}
		\frac{\Gamma \left(m-i+\frac{1-j-k}{2}\right)}{\Gamma \left(m-i-k+\frac{N}{2}\right)}=h(x') \delta_{j,k}.
	\end{align*}
\end{proof}

\begin{proof}[Proof of Theorem \ref{exp:sol:thm}]
	Let $u$ be given by \eqref{u:eden}. Then, $u$ satisfies \eqref{eq1} and \eqref{es1}, by Theorem \ref{thm:edenhofer}. It remains to show uniqueness.  Let $v\in C^{2s+\beta}(\Rp)$ denote a solution of \eqref{eq1} satisfying \eqref{es1} and set $z=u-v$.  Since $u$ and $v$ satisfy \eqref{reg:claim} and \eqref{es1} respectively, there is $K>0$ such that, for $x\in\{0\leq x_1\leq 1\}$,
	\begin{align*}
		|\partial_{1}^m[(\tau x_1)_+^{1-\sigma} z(\tau x_1,x')]|
		&\leq |\partial_{1}^{m-1}[(x_1)_+^{1-\sigma} u(x_1,x')]-h_m(x')|+|\partial_{1}^{m-1}[(x_1)_+^{1-\sigma} v(x_1,x')]-h_m(x')|\\
		&\leq 
		K \frac{x_1^{\beta}}{1+|x'|^N}.
	\end{align*}
	Then, since $D^{s-2}z=0$ on $\partial \Rp$,
	\begin{align*}
		\partial_{1}^{m-1}[(x_1)_+^{1-\sigma} z(x_1,x')]
		=\int_0^1 \partial_{1}^m[(\tau x_1)_+^{1-\sigma} z(\tau x_1,x')] x_1\ d\tau\leq K \frac{x_1^{1+\beta}}{1+|x'|^N}\quad \text{for }x\in\{0\leq x_1\leq 1\}.
	\end{align*}
	Iterating this argument, we obtain
	\begin{align}\label{pes}
		|z(x)|\leq K \frac{x_1^{s-1+\beta}}{1+|x|^{N}}\qquad \text{ for }x\in \{0\leq x_1\leq 1\}.
	\end{align}
	Thus, by \eqref{pes}, \eqref{lem:edenhofer-decay}, and \eqref{es1}, there is $C>K$ such that
	\begin{align*}
		|z(x)|\leq C \frac{\max\{x_1^{s-1+\beta},x_1^s\}}{1+|x|^{N}}\qquad \text{ for }x\in \Rp,
	\end{align*}
	but then $z\equiv 0$ in $\R^N$, by Lemma \ref{lem:bvpN}, and the uniqueness follows. 
\end{proof}

Observe that 
\begin{align*}
	E_{m,s}(x,y)
	=\frac{2}{\omega_N}\frac{(x_1)_+^{s}}{|x-y|^{N}}\quad \text{ and }\quad 
	E_{m-1,s}(x,y)
	=\frac{2N}{\omega_N}\frac{(x_1)_+^{s+1}}{|x-y|^{N+2}}\qquad \text{for $x\in \R^N$ and $y\in \partial \Rp$}.
\end{align*}
These kernels are connected via the trace operators with the Green function $\cG_s$,
see \eqref{first} and \eqref{second} below; however, the relationship between $E_{k,s}$ and $\cG_s$ is not so simple for $k\leq m-2$, see Remark \ref{ns:rem}.

\begin{lemma}\label{lem:edenhofer-rep}
	Let $m\in\N_0$, $\sigma\in(0,1]$, $s=m+\sigma$, $x\in \R^N$, and $z\in \partial \Rp$, then
	\begin{equation}\label{first}
		\frac{E_{m,s}(x,z)}{\Gamma(s+1)\Gamma(s) }= D^s[\cG_s(x,\cdot)](z)=\lim_{w\to z,\ w\in \Rp} \frac{\cG_s(x,w)}{w_1^s}
	\end{equation}
	and, if $m\geq 1$,
	\begin{equation}\label{second}
		\frac{\Gamma(s)\Gamma(s+2)}{1-s} D^{s+1}_z\cG_s(x,z)=\Gamma(s)\Gamma(s-1)D^{s-1}_{z}(-\Delta)_zG_s(x,z)]=E_{m-1,s}(x,z).
	\end{equation}
\end{lemma}
\begin{proof}
	We argue as in \cite[Lemma 4.1]{AJS16b}, let $z\in \partial \Rp$ and $x\in\Rp$, then
	\begin{align*}
		D^s[\cG_s(x,\cdot)](z)=k_{N,s}4^sx_1^{s}\lim_{w\to z,\ w\in \Rp} \int_0^1\frac{t^{s-1}}{(4x_1w_1 t+|x-w|^2)^{\frac{N}{2}}}\ dt=\frac{4^sk_{N,s}}{s}\frac{x_1^s}{|x-z|^{N}},
	\end{align*}
	where {$k_{N,s}:=\frac{\Gamma(\frac{N}{2})}{\pi^\frac{N}{2}4^s\Gamma(s)^2}$} and therefore
	\[
	\frac{4^sk_{N,s}}{s}=\frac{\Gamma(\frac{N}{2})}{\pi^{\frac{N}{2}}s\Gamma(s)^2}=\frac{\alpha_{0,0}}{\Gamma(s+1)\Gamma(s)}.
	\]
	For \eqref{second} let $s>1$, denote $\bar{x}=\overline{(x_1,x')}=(-x_1,x')$ for $x=(x_1,x')\in \R^N$, and let $z\in \partial\Rp$. Observe that $4(s-1)^2k_{N,s}=k_{N,s-1}$ and, by \eqref{leibniz},
	\[
	D^{m}_{z}(z_1^{m-1}\frac{x_1+z_1}{|z-\bar{x}|^N})
	=\sum_{k=0}^{m}D^{k}_{z}z_1^{m-1}D^{m-k}_{z}\Big(\frac{x_1+z_1}{|z-\bar{x}|^N}\Big)
	=D^{1}_{z}\Big(\frac{x_1+z_1}{|z-\bar{x}|^N}\Big)
	=|z-\bar{x}|^{-N}-N
	\frac{x_1^2}{|z-\bar{x}|^{N+2}}.
	\]
	As in the case of the ball (see \cite{AJS16b}) a direct calculation shows that the following recurrence formula holds:
	\begin{align*}
		-\Delta_x\ \cG_{s}(x,y)=\cG_{s-1}(x,y)-k_{N,s}4^s(s-1)P_{s-1}(x,y)\qquad \text{ for } x,y\in \Rp, \ x\neq y,
	\end{align*}
	where
	\begin{align}\label{Psm}
		P_{s-1}(x,y)&:=\frac{(x_1)_+^{s-2}(y_1)_+^{s-1}(x_1+y_1)}{|x-\bar{y}|^{N}}\qquad \text{ for }\quad x,y\in\R^N,\ x\neq y,
	\end{align}
	$k_{N,s}$ is as in \eqref{constants}, and $\bar{y}=(-y_1,y')$ for $(y_1,y')\in \R^N$. But then, by \eqref{first},
	\begin{align*}
		D^{s-1}_{z}(-\Delta)_zG_s(x,z)&=D^{s-1}_{z}G_{s-1}(x,z)-4^sk_{N,s}(s-1)D^{s-1}_{z}P_{s-1}(x,z)\\
		&=\frac{4^{s-1}k_{N,s-1}}{s-1}\frac{x_1^{s-1}}{|x-z|^{N}}-4^{s-1}\frac{k_{N,s-1}}{s-1}x_1^{s-1}D^{s-1}_{z}\Big(z_1^{s-2}\frac{x_1+z_1}{|z-\bar{x}|^N}\Big)\\
		&=\frac{4^{s-1}k_{N,s-1}}{s-1}\frac{x_1^{s-1}}{|x-z|^{N}}\Big(1-|x-z|^N D^{m}_{z}\Big(z_1^{m-1}\frac{x_1+z_1}{|z-\bar{x}|^N}\Big)\Big)\\
		&=\frac{4^{s-1}k_{N,s-1}}{s-1}\frac{x_1^{s-1}}{|x-z|^{N}} N\frac{x_1^2|x-z|^N}{|z-\bar{x}|^{N+2}}
		=\frac{N4^{s-1}k_{N,s-1}}{(s-1)}\frac{x_1^{s+1}}{|x-z|^{N+2}},
	\end{align*}
	where we used that $|x-\bar{z}|=|x-z|$ for $z\in \partial \Rp$, $z_1=0$. But then equation \eqref{second} follows, since
	\[
	\Gamma(s)\Gamma(s-1)N\frac{4^{s-1}k_{N,s-1}}{s-1}=N\frac{\Gamma(\frac{N}{2})}{\pi^{\frac{N}{2}}}= {\frac{2N}{\omega_N} = \alpha_{1,0}}.
	\]
	On the other hand, for any $z\in\partial\Rp$ we can also compute
	\begin{align*}
		D^{s+1}_zG_s(x,z)=\lim_{y\to z}\frac{\partial}{\partial y_1}[y_1^{-s}\cG_s(x,y)]
		=4^sk_{N,s}x_1^s\lim_{y\to z}\frac{\partial}{\partial y_1}\left[\frac1{{|x-y|}^{N}}\int_0^1\frac{v^{s-1}}{{v+1}^{\frac{N}2}}\;dv\right],
	\end{align*}
	where we have used an equivalent expression for $\cG_s$ to the one in \eqref{greenhs0}.
	Let us first remark that $\frac{2}{\omega_N}=4^s\Gamma(s)^2 k_{N,s}$ (see \eqref{constants}).
	Then,
	\begin{align*}
		\lim_{y\to z}&\frac{\partial}{\partial y_1}\left[\frac1{{|x-y|}^{N}}\int_0^1\frac{v^{s-1}}{\left(\frac{4x_1y_1}{|x-y|^2} v+1\right)^{\frac{N}2}}\;dv\right]\\
		&= 	\lim_{y\to z}\frac{\partial}{\partial y_1}\left[\frac1{{|x-y|}^{N}}\right]\int_0^1\frac{v^{s-1}}{\left(\frac{4x_1y_1}{|x-y|^2} v+1\right)^{\frac{N}2}}\;dv 
		+\lim_{y\to z}\frac1{{|x-y|}^{N}}\frac{\partial}{\partial y_1}\left[\int_0^1\frac{v^{s-1}}{\left(\frac{4x_1y_1}{|x-y|^2} v+1\right)^{\frac{N}2}}\;dv\right]\\
		&=	\frac1s\lim_{y\to z}\frac{\partial}{\partial y_1}\frac1{{|x-y|}^{N}}
		-\frac{N}2\frac1{{|x-z|}^{N}}\left[\int_0^1 v^s\lim_{y\to z}\frac{\partial}{\partial y_1}\frac{4x_1y_1}{|x-y|^2} \;dv\right],
	\end{align*}
	where 
	\begin{align*}
		\frac{\partial}{\partial y_1}\frac{4x_1y_1}{|x-y|^2}
		=\frac{4x_1}{|x-y|^2}+\frac{8x_1y_1(x_1-y_1)}{|x-y|^4}
		\longrightarrow\frac{4x_1}{|x-z|^2}	
		\qquad \text{as }\ y\to z.
	\end{align*}
	Thus, we have
	\begin{align*}
		D^{s+1}_z\cG_s(x,z)&=\frac{2\;x_1^s}{s\Gamma(s)^2\omega_N}\lim_{y\to z}\frac{\partial}{\partial y_1}\frac1{{|x-y|}^{N}}
		-\frac{4N}{(s+1)\Gamma(s)^2\omega_N}\frac{x_1^{s+1}}{{|x-z|}^{N+2}} \\
		&=\frac{2N}{s\Gamma(s)^2\omega_N}\frac{x_1^{s+1}}{{|x-y|}^{N+2}}
		-\frac{4N}{(s+1)\Gamma(s)^2\omega_N}\frac{x_1^{s+1}}{{|x-z|}^{N+2}} \\
		&=\left(\frac{1}{s\Gamma(s)^2}
		-\frac{2}{(s+1)\Gamma(s)^2}\right)\frac{2N}{\omega_N}\frac{x_1^{s+1}}{{|x-z|}^{N+2}}
		=\frac{1-s}{\Gamma(s)\Gamma(s+2)}E_{m-1,s}(x,z).
	\end{align*}
\end{proof}

\begin{remark}We close this paper with a series of remarks.
	\begin{enumerate}
		\item [$(i)$] If one defines the $s$-Martin kernel via
		\begin{equation*}
			M_s(x,z):=\lim_{\substack {w\to z,\\ w\in \Rp}} \frac{\cG_s(x,w)}{w_1^s}\quad\text{ for $x\in \Rp$ and  $z\in \partial \Rp$,}
		\end{equation*}
		where $\cG_s$ is as in \eqref{greenhs0}, then Lemma \ref{lem:edenhofer-rep} shows that $E_{m,s}$ is the Martin kernel for the half-space (up to a constant).
		\item[$(ii)$] Let $s>1$, $y\in \Rp$, and $P_{s-1}$ as in \eqref{Psm}, then 
		\[
		P_{s-1}(x,y)= \frac{\omega_N}{2}\int_{\partial \Rp} E_{m-1,s-1}(x,z)E_{m,s}(y,z)\ dz \qquad\text{ for $x\in \Rp$,}
		\]
		where $E_{k,s}$ are given in \eqref{edenhofer}. A similar relationship was observed in the ball, see \cite{AJS17a}.
		\item[$(iii)$] As in the ball case, the kernel $\Gamma_s$ satisfies a recurrence formula (see \cite[Proposition 3.1]{AJS17a}). Let $m\in\N$, $\sigma\in(0,1)$, $s=m+\sigma$, and $\Gamma_s$ as in \eqref{s-poisson}.Then
		\begin{equation*}
			\Gamma_s(x,y)=\Gamma_{s-1}(x,y)-\int_{\partial\Rp} E_{m,s}(x,z) D^{s-1}\Gamma_{s-1}(z,y)\ dz\qquad\text{ for } x\in \Rp,\ y\in \R^N\setminus \overline{\Rp}.
		\end{equation*}
		\item [$(iv)$] As a consequence of Theorem \ref{dist:sol:l} and Lemma \ref{lem:edenhofer-rep} one can deduce a higher-order Hopf Lemma for the homogeneous Dirichlet problem as in \cite[Theorem 5.7]{GGS10} and \cite[Corollary 1.9]{AJS17a} where the case of a ball is considered and the positivity of the trace $D^s u$ on the boundary of the domain can be obtained if $(-\Delta)^s u$ is positive.  Note however that in the half-space one requires additional growth assumptions at infinity for the solution (such as \eqref{bound}, \eqref{es}, or \eqref{es1}).  Without these assumptions, the function $x_1^s$ (which is $s$-harmonic by Proposition \ref{s-harmonic2}) would be a counterexample.
		\item[$(v)$] Similar arguments to those presented for the half-space can be used to deduce the kernels for the complement of the ball $B^c:=\{x\in\R^N\::\:|x|>1\}$ using the Kelvin transform ($K_s$ as in \eqref{kappa:K:def} with $c=1$ and $v=0$).  Here, similarly as in the ball case (see \cite{AJS17a}), for $k\in\N_0$ the \emph{boundary trace operator for $B^c$} is given by
		\begin{equation*}
			D_{B^c}^{k+\sigma-1}u(z):=\frac{(-1)^k}{k!}\lim_{x\to z}\frac{\partial^k}{\partial (|x|^2)^k}[{(|x|^2-1)}^{1-\sigma} u(x)]\qquad \text{ for }z\in\partial B.
		\end{equation*}
		The associated Green function and nonlocal Poisson kernels are 
		\begin{align*}
			\cG_{B^c}(x,y)&=k_{N,s}|x-y|^{2s-N}\int_0^{\frac{(|x|^2-1)_+(|y|^2-1)_+}{|x-y|^2}}\frac{v^{s-1}}{(v+1)^{\frac{N}{2}}}\ dv,\quad x,y\in \R^N,\ x\neq y,\\
			\Gamma^{B^c}_s(x,y)&=(-1)^m\gamma_{N,\sigma}\frac{(|x|^2-1)_+^s}{(1-|y|^2)^s|x-y|^N}\quad x\in\R^N,\ y\in B. 
		\end{align*}
		see also \cite[equation (8.6)]{BB00} and \cite[equation (1.6.11)]{L72} for the case $s\in(0,1)$.  Boundary Poisson kernels for $B^c$ can also be obtained via the Kelvin transform using the boundary Poisson kernels for the ball given in \cite{AJS17a}.
	\end{enumerate}
\end{remark}

\begin{remark}\label{ns:rem}
	For $m\in \N$, $s=m+1$, $k\in\{0,\ldots,m\}$, $x\in\Rp$, and $y\in\partial\Rp$ let
	\begin{align*}
		K_{j,m+1}(x,y):=\left\lbrace\begin{aligned}
			& \partial_{y_1} (-\Delta)_y^{m+1-(\frac{j}{2}+1)}\cG_{m+1}(x,y) & \text{ for $j$ even,}\\
			& (-\Delta)_y^{m+1-(\frac{j+1}{2})}\cG_{m+1}(x,y) & \ \text{ for $j$ odd.}\\
		\end{aligned}\right.
	\end{align*}
	Due to integration by parts (see, for example, \cite[Lemma 8]{RW09}), one could think that the Dirichlet boundary Poisson kernels satisfy the relationships $E_{j,m+1}(x,y) ={ c_{j,m}}K_{j,m+1}(x,y)$ for $x\in\Rp$, $y\in\partial\Rp$, $j\in\{0,\ldots,m\}$,
	and some constants $c_{j,m}\in\R$.  Although this equality holds true for $j\in\{m-1,m\}$ (see Lemma \ref{lem:edenhofer-rep}), it does \emph{not} hold in general for $k\leq m-2$. For instance, let $m=2$, $\sigma=1$, $s=3$, and $N=2$.  Direct calculations show that
	\begin{align*}
		K_{0,3}(x,y)=\partial_{y_1} (-\Delta)_y^{2}\cG_{3}(x,y)
		=\frac{8}{\pi}\frac{x_1^{5}}{|x-y|^{6}}-\frac{4}{\pi}\frac{x_1^3}{|x-y|^{4}}\qquad \text{ for $x\in{\R^2_+}$ and $y\in\partial\R^2_+$},
	\end{align*}
	whereas,
	\begin{align*}
		E_{0,3}(x,y)=\frac{4}{\pi}\frac{x_1^{5}}{|x-y|^{6}}-\frac{1}{\pi}\frac{x_1^3}{|x-y|^{4}}\qquad \text{ for $x\in{\R^2_+}$ and $y\in\partial{\R^2_+}$}.
	\end{align*}
	In fact, for $h\in C^\infty_c(\R)$ and arguing as in Theorem \ref{thm:edenhofer}, one can show that the function $u:\R^2_+\to\R$ given by $u(x)=\int_{\R}K_{0,3}(x,y)h(y)\ dy$ satisfies
	\begin{align*}
		u(0,x_2) = h(x_2),\qquad \partial_{x_1} u(0,x_2) = 0,\qquad \partial_{x_1x_1} u(0,x_2) = -h''(x_2)\qquad \text{ for all }x_2\in\R,
	\end{align*}
	that is, Dirichlet boundary conditions are not met. However, it does hold that
	\begin{align*}
		\Delta u(0,x_2) = 0 \qquad \text{ for all }x_2\in\R.
	\end{align*}
	
	Actually, a closer look at integration by parts formulas \cite[Lemma 8]{RW09} suggests that, for $j\in\{0,\ldots,m\}$, the function $u_j:\Rp\to\R$ given by $u_j(x)=\int_{\partial\Rp} K_{j,m+1}(x,y)h(y)\ dy$  is a solution of $(-\Delta)^{m+1} u_j=0$  in $\Rp$ and, for $k\in\{0,\ldots,m\}$,
	\begin{align*}
		(-\Delta)^{\frac{k}{2}}u_j=\delta_{j,k}h\quad \text{ for $k$ even},\qquad 
		\partial_{x_1}(-\Delta)^{\frac{k-1}{2}}u_j=\delta_{j,k}h\quad \text{ for $k$ odd}\qquad \text{ on }\partial\Rp.
	\end{align*}
	The relationship between $E_{k,s}$ and $\cG_s$ does not seem to be simple in general and therefore a generalization of our results to more general domains is not immediate.
\end{remark}

\paragraph{{\bf Acknowledgment}} We thank M. Kassmann for informing us about \cite{GKL18} and the issue mentioned in Remark \ref{remark:crucial}.

\end{document}